\documentclass{amsart}
\usepackage{latexsym,amsmath,amssymb,amsfonts,amscd,graphics}
\usepackage[mathscr]{eucal}
\usepackage[all]{xypic}
\usepackage[normalem]{ulem}

\numberwithin{equation}{section}
\theoremstyle{plain}
\newtheorem{theorem}[equation]{Theorem}
\newtheorem*{mainthm}{Main Theorem}

\newtheorem{lemma}[equation]{Lemma}

\newtheorem{proposition}[equation]{Proposition}
\newtheorem{corollary}[equation]{Corollary}

\theoremstyle{remark}
\newtheorem{remark}[equation]{Remark}

\theoremstyle{definition}
\newtheorem{definition}[equation]{Definition}
\newtheorem{notation}[equation]{Notation}

\newcommand{\bP}{\mathbb{P}}

\newcommand{\PP}{\mathbb{P}}

\newcommand{\bQ}{\mathbb{Q}}

\newcommand{\bC}{\mathbb{C}}

\newcommand{\calO}{\mathscr{O}}
\newcommand{\OO}{\mathcal{O}}

\newcommand{\Pic}{\operatorname{Pic}}
\newcommand{\Hilb}{\operatorname{Hilb}}

\newcommand{\SL}{\mathrm{SL}}
\newcommand{\SO}{\mathrm{SO}}
\newcommand{\Sing}{\mathrm{Sing}}
\newcommand{\Chow}{\operatorname{Chow}}

\newcommand{\rank}{\mathrm{rank}}
\newcommand{\gquot}{/\!\!/}

\DeclareMathOperator{\Proj}{Proj}

\title[VGIT for genus $4$ curves]{Log canonical models and variation of GIT for genus four canonical curves}
\author[S. Casalaina-Martin]{Sebastian Casalaina-Martin } \address{University of Colorado, Department of Mathematics, Boulder, CO 80309-0395} \email{casa@math.colorado.edu}
\author[D. Jensen]{David Jensen}
\address{Stony Brook University, Department of Mathematics,  Stony Brook, NY 11794}
\email{djensen@math.sunysb.edu}

\author[R. Laza]{Radu Laza}
\address{Stony Brook University, Department of Mathematics,  Stony Brook, NY 11794}
\email{rlaza@math.sunysb.edu}

\thanks{The  first author was partially supported by NSF grant DMS-1101333.  The  third author was partially supported by NSF grant DMS-0968968 and a Sloan Fellowship}

\date{\today}

\begin{document}
\begin{abstract}
We discuss GIT for canonically embedded genus four curves and the connection to the Hassett--Keel program. A canonical genus four curve is a  complete intersection of a quadric and a cubic, and, in contrast to the genus three case, there is a family of GIT quotients that depend on a choice of linearization. We discuss  the corresponding VGIT problem and show that the resulting spaces give the final steps in the Hassett--Keel program for genus four curves.
\end{abstract}

\bibliographystyle{alpha}
\maketitle

\section*{Introduction}\label{introlabel}
The Hassett--Keel program aims to give modular interpretations of certain log canonical models of $\overline M_g$, the moduli space of stable curves of fixed genus $g$, with the ultimate goal of giving a modular interpretation of the canonical model for the case $g\gg 0$.   The program, while relatively new, has attracted the attention of a number of researchers, and has rapidly become one of the most active areas of research concerning the moduli of curves.   Perhaps the most successful approach so far has been to compare these log canonical models to alternate compactifications  of $M_g$ constructed  via GIT on the spaces $\Hilb_{g,\nu}^m$, the so-called $m$-th Hilbert spaces of $\nu$-canonically embedded curves of genus $g$,  for ``small'' $\nu$ and $m$ (e.g.~\cite{hh}, \cite{hh2},  \cite{ha}).

\smallskip

For large genus, completing the program in its entirety still seems somewhat out of reach.
On the other hand, the case of low genus curves affords  a  gateway to the general case, providing motivation and corroboration of expected behavior.   The genus $2$ and $3$ cases were completed  recently (\cite{hg2}, \cite{hl}).
 In this paper, we study the genus $4$ case by focusing on the spaces $\operatorname{Hilb}_{4,1}^m$; i.e.~we study GIT quotients of canonically embedded genus  $4$  curves. The main result is a complete description of GIT stability on $\Hilb_{4,1}^m$ for  all $m$, as well as a proof  that  the resulting GIT quotients give the final steps in the Hassett--Keel program for genus $4$.  Together with previous work on the subject (see \cite{hl2}, \cite{maksymg4}, \cite{g4ball}),  this completes  the program in genus $4$ outside of a small range.

\smallskip

One of the key features of this paper is the technique employed.
Using a space we denote by $\bP E$ (a smooth, elementary, birational model of the Hilbert scheme parameterizing complete intersections) we fit all of the Hilbert quotients for canonical genus  $4$  curves into  a single variation of GIT problem (VGIT). In other words,  the final steps of the Hassett--Keel  program in genus  $4$  are described by a VGIT problem on a single space.    Also of interest is a technical point that arises: we are forced to do VGIT for linearizations that lie outside of the ample cone.  \emph{A priori} this leads to an ambiguity in the meaning of Mumford's numerical criterion for stability. However we are able to circumvent this issue to provide a complete analysis of the stability conditions on $\bP E$.

\smallskip

While examples of GIT for hypersurfaces are abundant in the literature (e.g. \cite[\S4.2]{GIT}, \cite{shah}, \cite{allcock1}, \cite{laza}), this appears to be one of the first examples of GIT for complete intersections (see however \cite{avmir} and \cite{mm} for $(2,2)$ complete intersections, and Benoist \cite{benoist} for some generic stability results in a situation similar to ours).  Furthermore, unlike the projective spaces parameterizing hypersurfaces or the Grassmannian parameterizing complete intersections of type $(d,\dots,d)$, the natural parameter space in our situation has Picard rank two, and thus provides a natural setting for variation of GIT.  We believe the techniques we develop in this paper for studying VGIT for spaces of complete intersections will have a number of further applications beyond moduli spaces of curves.

\subsection*{The Hassett--Keel program for genus  $4$  curves: known and new results}
To put our results in context, we recall some background on  the Hassett--Keel program.
Namely, for $\alpha\in [0,1]$, the log minimal models of $\overline{M}_g$ are defined to be the projective varieties
$$ \overline{M}_g ( \alpha ) := \Proj \left( \bigoplus_{n=0}^{\infty} H^0 \left(n(K_{\overline{M}_g} + \alpha \delta ) \right) \right), $$
where $\delta$ is the boundary divisor in $\overline{M}_g$. Hassett and Hyeon have explicitly constructed the log minimal models $\overline{M}_g ( \alpha )$ for $\alpha \geq \frac{7}{10} - \epsilon$ (see \cite{hh,hh2}).  Hyeon and Lee have also described the next stage of the program in the specific case that $g=4$ (see \cite{hl2}): as $\alpha$ decreases from $\frac{2}{3}+\epsilon$ to $\frac{2}{3}$, they construct a map that contracts the locus of Weierstrass genus 2 tails, replacing them with $A_4$ singularities. Thus, the known spaces for the Hassett--Keel program in genus $4$ are:
\begin{equation}\label{eqknown}
\xymatrix@R=.15in @C=.15in{
\overline{M}_4 = \overline{M}_4 [1,\frac{9}{11}) \ar@{=>}[d] & & & \\
\overline{M}_4^{ps} = \overline{M}_4 [ \frac{9}{11},\frac{7}{10} ) \ar[rd] \ar@{-->}[rr]& & \overline{M}_4^{hs} = \overline{M}_4 ( \frac{7}{10} , \frac{2}{3} ) \ar[ld] \ar[rd] & \\
 & \overline{M}_4^{cs} = \overline{M}_4 ( \frac{7}{10} ) & & \overline{M}_4 ( \frac{2}{3} )  }
 \end{equation}
 where the notation $\overline M_g(I)$ for an interval $I$ means $\overline M_g(\alpha)\cong \overline M_g(\beta)$ for all $\alpha,\beta \in I$.  The double arrows correspond to divisorial contractions, the single arrows to small contractions, and the dashed arrows to flips.

\smallskip

The main result of the paper is the construction of the log minimal models $\overline{M}_4 ( \alpha )$ for $\alpha \leq \frac{5}{9}$ via a VGIT analysis of canonically embedded curves in $\bP^3$.

\begin{mainthm}
For $\alpha \le \frac{5}{9}$, the log minimal models $\overline{M}_4(\alpha)$ arise as GIT quotients of the parameter space $\bP E$.  Moreover, the VGIT problem gives us the following diagram:

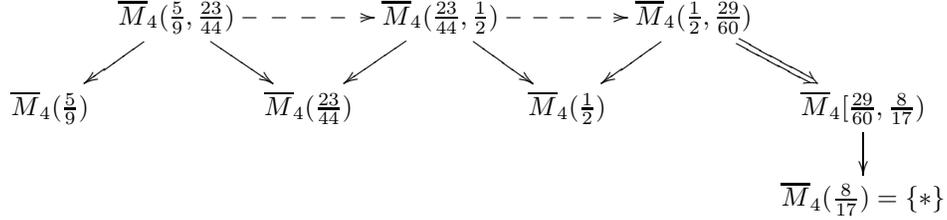
\begin{equation}\label{mainthm}
\xymatrix @R=.07in @C=.07in{
 & \overline{M}_4 ( \frac{5}{9} , \frac{23}{44} ) \ar@{-->}[rr]\ar[ldd] \ar[rdd] & & \overline{M}_4 ( \frac{23}{44} , \frac{1}{2} ) \ar[ldd] \ar[rdd] \ar@{-->}[rr]& & \overline{M}_4 ( \frac{1}{2} , \frac{29}{60} ) \ar[ldd] \ar@{=>}[rdd] & \\
&&&&&& \\
\overline{M}_4 ( \frac{5}{9} ) & & \overline{M}_4 ( \frac{23}{44} ) & & \overline{M}_4 ( \frac{1}{2} ) & & \overline{M}_4 [\frac{29}{60},\frac{8}{17}) \ar[dd] \\
&&&&&&\\
 & & & & & & \overline{M}_4 (\frac{8}{17} )   = \{*\}
 }
 \end{equation}

 More specifically,
 \begin{itemize}
 \item[i)] the end point $\overline{M}_4 ( \frac{8}{17}+\epsilon)$ is obtained via GIT for $(3,3)$ curves on $\bP^1 \times \bP^1$ as discussed in \cite{maksymg4};
 \item[ii)] the other end point $\overline{M}_4 ( \frac{5}{9} )$ is obtained via GIT for the Chow variety of genus 4 canonical curves as discussed in \cite{g4ball};
 \item[iii)] the remaining spaces   $\overline{M}_4 ( \alpha ) $ for $\alpha$ in the range $\frac{8}{17} < \alpha<  \frac{5}{9}$  are obtained via appropriate $\Hilb^m_{4,1}$ quotients, with the exception of $\alpha= \frac{23}{44}$.
 \end{itemize}
\end{mainthm}

Thus in genus  $4$, the remaining unknown range for the Hassett--Keel program is the interval $\alpha\in ( \frac{5}{9} , \frac{2}{3} )$.  Using the geometric meaning of the spaces  $\overline{M}_4 (\alpha)$ for $\alpha\le \frac{5}{9}$ and the  predictions of \cite{afs}, we expect  that there are exactly two more critical values:  $\alpha = \frac{19}{29}$, when   the divisor $\delta_2$ should be contracted to a point, and $\alpha = \frac{49}{83}$, when   the locus of curves with hyperelliptic normalization obtained by introducing a cusp at a Weierstrass point should be flipped, being replaced by curves with $A_6$ singularities.  We do not expect that these models can be obtained by further varying the GIT problem we consider here.  In fact, since each of these predicted models arises prior to the predicted flip of the hyperelliptic locus ($\alpha=\frac{5}{9}$), they should be unrelated to spaces of canonical curves.  It is believed that each of these two intermediate models ought to correspond to a quotient of the Hilbert scheme of bicanonical curves.

\subsection*{GIT for canonical genus  $4$  curves} As already mentioned, GIT for pluricanonical  curves has  long been  used to produce projective  models for the moduli space of curves. For example Mumford used asymptotic stability for $\nu$-canonical curves, with $\nu \geq 5$, to show the projectivity of $\overline M_g$, and recently the case $\nu<5$ has been used  in the Hassett--Keel program.  The basic idea is that as the values $\nu$ and $m$ decrease one should obtain the log minimal models $\overline{M}_g ( \alpha )$ for progressively smaller values of $\alpha$ (e.g.~\cite[Table 1]{fsmoduli}). Thus from the perspective of the Hassett--Keel program, it is of interest to understand GIT for canonically embedded curves.
This turns out to be difficult, and to our knowledge the only case where the stability conditions have been described completely prior to this paper is for genus $3$.  On the other hand, it was recently proved (see \cite{afs11-1}) that the generic non-singular canonical curve of arbitrary genus is stable.  In this paper, we completely describe the stability conditions for genus $4$ canonical curves.

\smallskip

We set up the analysis of the GIT stability for canonical genus $4$ curves as follows. The canonical model of a smooth, non-hyperelliptic genus  $4$  curve  is a $(2,3)$-complete intersection in $\bP^3$. A natural parameter space for complete intersections is a projective bundle $\bP E \to \bP^9$ on which $G=\SL(4,\bC)$ acts naturally. Since $\rank(\Pic (\bP E))=2$, the GIT computation involves a choice of linearization parameterized by  $t\in\bQ_+\cup \{0\}$  (corresponding to the linearization $\eta+t h$, where $\eta$ is the pullback of $\calO(1)$ from the space of quadrics $\bP^9$  and $h$ is the relative $\calO(1)$). In this paper, we analyze the geometry of the quotients $\bP E\gquot_t \SL(4)$ as the linearization varies and relate them to the Hassett--Keel spaces $\overline M(\alpha)$.  We note that a related setup for GIT for complete intersections occurs in recent work of Benoist \cite{benoist}.

\smallskip

One naturally identifies two special cases. First, for $0<t\ll 1$ one easily sees that $\bP E\gquot_t \SL(4)$ coincides with the  GIT quotient for $(3,3)$ curves on $\bP^1\times \bP^1$; this was analyzed  by Fedorchuk \cite{maksymg4}. At the other extreme, the case $t=\frac{2}{3}$ was shown in \cite{g4ball} to be isomorphic to both the quotient of the Chow variety for genus  $4$  canonical curves,  as well as to the Hassett--Keel space $\overline M_4(\frac{5}{9})$. The content of this paper is to describe the GIT quotient for the intermediary values $t\in(0,\frac{2}{3})$. We work with $\bP E$, but  show that all of the  quotients of type $\Hilb^m_{4,1}\gquot\SL(4)$ arise in this way.
The advantage of working with $\mathbb PE$ is that we have VGIT on a single, elementary space, where the stability computation is  straightforward and  corresponds directly to the variation of parameters.

\subsection*{Geometric description of the birational maps in the main theorem} As mentioned, the Hassett--Keel program aims to give modular interpretations to the spaces $\overline M_g(\alpha)$ and to the birational maps between them. Essentially, as $\alpha$ decreases, it is expected that $\overline M_g(\alpha)$ parameterizes curves with increasingly complicated singularities, and at the same time special curves from $\overline M_g$ are removed (e.g. curves with elliptic tails, or hyperelliptic curves, etc.). In the situation of our main result, the maps of the diagram \eqref{mainthm} are intuitively described as follows.  In $\overline{M}_4 ( \frac{5}{9} )$, the hyperelliptic locus is contracted, as is the locus of elliptic triboroughs, and the locus of curves in $\Delta_0$ with hyperelliptic normalization obtained by gluing two points that are conjugate under the hyperelliptic involution.  The next map flips these loci, replacing them with curves that have $A_8$, $D_4$, and $A_7$ singularities, respectively.

\smallskip

The second flip (at $\alpha=\frac{23}{44}$) removes the locus of cuspidal curves whose normalization is hyperelliptic, replacing them with curves possessing a separating $A_7$ singularity.  The third flip  (at $\alpha=\frac{1}{2}$) removes the locus of nodal curves whose normalization is hyperelliptic, replacing them with the union of a conic and a double conic.  Finally, the map to $\overline{M}_4 ( \frac{29}{60} )$ contracts the Gieseker--Petri divisor to a point, corresponding to a triple conic.  This geometric description of the various maps is summarized in Tables \ref{table1} and \ref{table2} in \S \ref{sectboundary}.

\smallskip

We note that the critical slopes occurring in our analysis are in concordance with the general predictions of Alper--Fedorchuk--Smyth \cite{afs}. We also note that at $\alpha = \frac{5}{9}$ and $\alpha = \frac{23}{44}$, we observe a phenomenon that first occurs in genus $4$.  Namely, the critical values at which the separating $A_5$ and $A_7$ singularities appear differ from those at which the non-separating singularities appear.

\subsection*{Acknowldegements}
The authors are grateful to O.~Benoist and M.~Fedorchuk for discussions relevant to this paper, and for specific comments on an earlier draft.  We also thank the referees for detailed comments that  have improved the paper.

\subsection*{Table of spaces}
The following table, relating the parameters $\alpha$, $t$ and $m$,  describes the relationships among the various spaces occurring in this paper. Note that the following relations (see Proposition \ref{comparepehilb} and Theorem \ref{teopehk}) hold:
$$t = \frac{34 \alpha - 16}{33 \alpha - 14}, \ \ \alpha=\frac{14t-16}{33t-34}, \ \
t=\left\{
\begin{array}{ll}
\frac{m-2}{m+1} & 2\le m\le 4, \\
&\\
\frac{2m^2-8m+8}{3m^2-9m+8}& m\ge 5.\\
\end{array}
\right.
$$

\begin{table}
$$
\renewcommand\arraystretch{1.4}
\begin{array}{|c|c|c|c|c|}
\hline
 \overline M_g(\alpha)&\mathbb PE\gquot_t \SL(4)&m&\operatorname{Hilb}_{4,1}^m\gquot_{\Lambda_m} \SL(4)& \text{Other}\\
\hline \hline
\frac{5}{9}&\frac{2}{3}&\infty &-&\operatorname{Chow}_{4,1}\gquot_{\Lambda_\infty}\SL(4)\\
\hline
\left(\frac{23}{44},\frac{5}{9}\right)&\left(\frac{6}{11},\frac{2}{3}\right)& \left(\frac{17+\sqrt{129}}{4},\infty\right) &  \ge 8&\\
\hline
\frac{23}{44}&\frac{6}{11}&\frac{17+\sqrt{129}}{4}&-&\\
\hline
\left(\frac{1}{2},\frac{23}{44}\right)&\left(\frac{2}{5},\frac{6}{11}\right)&\left(4,\frac{17+\sqrt{129}}{4}\right) &5,6,7&\\
\hline
\frac{1}{2}&\frac{2}{5}&4& 4 &\\
\hline
\left(\frac{29}{60},\frac{1}{2}\right)&\left(\frac{2}{9},\frac{2}{5}\right) & \left(\frac{20}{7},4\right)& 3&\\
\hline
\left(\frac{8}{17},\frac{29}{60}\right]&\left(0,\frac{2}{9}\right] & \left(2,\frac{20}{7}\right] &-&\left|\mathscr O_{\mathbb P^1\times \mathbb P^1}(3,3)\right|\gquot \SO(4)\\
\hline
\frac{8}{17}&0 & 2&2&\\
\hline
\end{array}
$$
\vskip .2 cm
\caption{Relationship among the spaces appearing in this paper.}\label{tablecon}
\end{table}

\section{$\mathbb PE$ and its geometry}
In this section we recall the projective bundle $\mathbb PE$ considered in \cite{g4ball} (see also \cite{benoist} for a more general setup) parameterizing subschemes of $\mathbb P^3$ defined by a quadric and a cubic.  The primary aim is to describe various rational maps from $\mathbb PE$ to projective space and their induced polarizations in terms of standard generators for the Picard group.

\subsection{Preliminaries}
\label{Preliminaries}

We start by recalling the definition of the bundle $\mathbb PE$   from \cite{g4ball} and establishing some basic properties.   We fix the notation
$$
V_d:=H^0(\mathbb P^3,\mathscr O_{\mathbb P^3}(d))
$$
for each $d\in \mathbb Z$, and  define $\mathcal Q$ to be the universal quadric:
$$
\begin{CD}
\mathcal Q @>>> \mathbb P^3\times \mathbb P V_2\\
@VVV @VV \pi_2 V\\
\mathbb P V_2 @= \mathbb PV_2.
\end{CD}
$$
  There is an exact sequence of sheaves
\begin{equation}\label{eqnuniquad}
0\to \mathcal I_{\mathcal Q}\to \mathscr O_{\mathbb P^3\times \mathbb P V_2}\to \mathscr O_{\mathcal Q}\to 0.
\end{equation}
 Setting $\pi_1: \mathbb P^3\times \mathbb PV_2\to \mathbb P^3$ (resp. $\pi_2: \mathbb P^3\times \mathbb PV_2\to \mathbb PV_2$) to be the first (resp. second) projection, then   tensoring \eqref{eqnuniquad} by $\pi_1^*\mathscr O_{\mathbb P^3}(3)$ and projecting with $\pi_{2*}$ we obtain an exact sequence:
\begin{equation}\label{eqndfnE}
0\to \pi_{2 *}(\mathcal I_{\mathcal Q}\otimes \pi_1^*\mathcal O_{\mathbb P^3}(3))\to \pi_{2 *}(\pi_1^*\mathcal O_{\mathbb P^3}(3))\to \pi_{2*}(\mathscr O_{\mathcal Q}\otimes \pi_1^*\mathcal O_{\mathbb P^3}(3))\to 0.
\end{equation}

We will define the projective bundle $\mathbb PE$ using the locally free sheaf on the right.

\begin{definition}
In the notation above, let  $\mathscr E:=\pi_{2*}(\mathscr O_{\mathcal Q}\otimes \pi_1^*\mathcal O_{\mathbb P^3}(3))$, $E:=\underline{\operatorname{Spec}}_{\mathbb PV_2}(\mathscr E^\vee)$
  and $\mathbb P E:=\underline{\operatorname{Proj}}_{\mathbb PV_2}(\mathscr E^\vee)$.  We denote the natural projection as $\pi:\mathbb PE \to \mathbb PV_2$.
\end{definition}

\begin{remark}\label{remdefu}
Points of $\mathbb PE$ correspond to pairs $([q],[f])$ where $[q]\in \mathbb PV_2$ is the class of a non-zero  element  $q\in V_2$, and $[f]\in \mathbb PE_{[q]}$ is the class of a non-zero element $f\in V_3/\langle x_0 q, x_1q,x_2q,x_3q\rangle$.  Sometimes we will instead   consider $f$ as an element of $V_3$ not lying in the span of $\langle x_0 q, x_1q,x_2q,x_3q\rangle$.  We will often write $(q,f)$ rather than $([q],[f])$ if there is no chance of confusion.
This description motivates  calling $\mathbb PE$ \emph{the space of $(2,3)$-subschemes in $\mathbb P^3$}.  Throughout, we will write $U \subset \mathbb PE$ for the open subset of points $(q,f)$ such that $q$ and $f$ do not have a common factor.  Note there is a non-flat family of sub-schemes of $\mathbb P^3$  over $\mathbb PE$ that restricts to a flat family over $U$.
\end{remark}

We point out that
\begin{equation}\label{eqndfnEsimp}
\pi_{2 *}(\pi_1^*\mathcal O_{\mathbb P^3}(3))\cong V_3\otimes _{\mathbb C}\mathscr O_{\mathbb PV_2}
\ \ \text {and }  \ \
\pi_{2\ast}\left(\mathcal I_{\mathcal Q}\otimes \pi_1^*\mathcal O_{\mathbb P^3}(3))\right)\cong V_1 \otimes_{\mathbb C}  \mathscr O_{\mathbb P V_2}(-1),
\end{equation}
so \eqref{eqndfnE} can be written as
\begin{equation}\label{eqnefp} 0\to V_1 \otimes_{\mathbb C}  \mathscr O_{\mathbb P V_2}(-1)\to V_3\otimes _{\mathbb C}\mathscr O_{\mathbb PV_2} \to \mathscr E\to 0. \end{equation}

\begin{remark}\label{reminvts} With this description  of $\mathscr E$, it is easy to describe many of the invariants of $E$ and $\mathbb PE$.
Setting $x=c_1(\mathscr O_{\mathbb PV_2}(1))$, the Chern character of $E$ is $
\operatorname{ch(E)}=20-4\sum_{k=0}^\infty \frac{(-1)^kx^k}{k!}$.
Denoting the line bundles  $\eta=\pi^*\mathscr O_{\mathbb P V_2}(1)$ and $h=\mathscr O_{\mathbb PE}(1)$,  it is standard that $\Pic(\bP E)\cong \mathbb Z \eta \oplus \mathbb Z h$,
and $$K_{\mathbb PE}=-14\eta-16h. $$  We define the  \textbf{slope} of a line bundle $a\eta+bh$ (with $a\ne 0$) to be equal to $t=\frac{b}{a}$.
\end{remark}

\subsection{Morphisms to projective space}

As mentioned above, there is a family
$$
\begin{CD}
\mathcal C @>>> \mathbb P^3\times \mathbb P E\\
@VVV @VV \pi_2 V\\
\mathbb PE @= \mathbb PE.
\end{CD}
$$
of  $(2,3)$-subschemes of $\mathbb P^3$ parameterized by $\mathbb PE$ that is flat exactly over the locus $U$ of points $([q],[f])$ such that $q$ and $f$ do not have a common linear factor.
Consequently,
there is a birational map
$$
\mathbb PE \dashrightarrow \operatorname{Hilb}_{4,1}
$$
whose restriction to $U$ is a morphism; here  $\operatorname{Hilb}_{4,1}$ is the component of the Hilbert scheme containing genus $4$ canonical curves.

  \subsubsection{The moduli space of curves}
  \label{ModuliSpace}
The rational map $\operatorname{Hilb}_{4,1}\dashrightarrow \overline M_4$ induces a rational map
$$
\mathbb PE \dashrightarrow \overline M_4.
$$
Setting $\lambda$ and $\delta$ to be the pull-backs of the corresponding classes on $\overline M_4$ one can check  (e.g. \cite[\S 1]{g4ball}) that
\begin{eqnarray*}
\lambda &=& 4\eta+4 h,\\
\delta &=&33\eta+ 34 h.
\end{eqnarray*}
Conversely, $\eta=\frac{17}{2} \lambda-\delta$ and $h=-\frac{33}{4} \lambda+\delta$.

\subsubsection{Grassmannians}
 For each point in $\operatorname{Hilb}_{4,1}$, we have an associated ideal sheaf $\mathcal I\subseteq \mathscr O_{\mathbb P^3}$.
The generic point of $\operatorname{Hilb}_{4,1}$ corresponds to a canonical curve, so that $\mathcal I$ is the sheaf associated to a homogeneous ideal of the form $(q,f)\subseteq \mathbb C[X_0,\ldots,X_3]$ where $q$ is a quadric and $f$ is a cubic.
Since $q$ and $f$ have no common irreducible factors in this case, we get the following resolution of the ideal sheaf  $\mathcal I $:
$$
0\longrightarrow \mathscr O_{\mathbb P^3}(-5)\stackrel{(f,-q)}{\longrightarrow } \mathscr O_{\mathbb P^3}(-2)\oplus \mathscr O_{\mathbb P^3}(-3)\stackrel{\binom{q}{f}}{\longrightarrow } \mathcal I \longrightarrow 0.
$$
It follows that $$ k_m:=h^0(\mathcal I(m))= \binom {m+1}{3}+\binom{m}{3}-\binom{m-2}{3}. $$ Set $$ n_m=h^0(\mathscr O_{\mathbb P^3}(m))=\binom{m+3}{3}. $$
With this notation, there is a rational map $\psi_m:\operatorname{Hilb}_{4,1}\dashrightarrow \mathbb G(k_m,n_m)$, and recall that $\operatorname{Hilb}_{4,1}^m$ is defined to be the closure of the image of $\psi_m$.
  The Pl\"ucker embedding induces a linearization $\Lambda_m$ on $\operatorname{Hilb}_{4,1}^m$.
Composing  the rational map $\mathbb PE \dashrightarrow \operatorname{Hilb}_{4,1}$ with $\psi_m$ defines a rational map
$$\varphi_m:\mathbb PE\dashrightarrow \operatorname{Hilb}_{4,1}^m$$
that restricts to a morphism on the open set $U\subseteq\mathbb PE$.

Since $\mathbb PE$ is smooth, and $\operatorname{codim}_{\mathbb PE}(\mathbb PE \setminus U)\ge 2$, each line bundle on $U$ has a unique extension to a line bundle on $\mathbb PE$; in other words, the restriction map $\operatorname{Pic}(\mathbb PE)\to \operatorname{Pic}(U)$ is an isomorphism.
Since the restriction of $\varphi_m$ to $U$ is regular, there is a well defined pull-back $$\varphi_m^*:\operatorname{Pic}(\operatorname{Hilb}^m_{4,1})\to \operatorname{Pic}(\mathbb PE)$$ given by the composition $\operatorname{Pic}(\operatorname{Hilb}^m_{4,1})\xrightarrow{(\varphi_m|_U)^*}\operatorname{Pic}(U)\to  \operatorname{Pic}(\mathbb PE)$.

\begin{proposition}
\label{comparepehilb}
For all $m\in \mathbb N$
there is a rational map
$$
\varphi_m:\mathbb PE \dashrightarrow \operatorname{Hilb}_{4,1}^m
$$
that restricts to a morphism on the open set $U\subseteq\mathbb PE$ of points $([q],[f])$ such that $q$ and $f$ do not have a common linear factor.  The pull-back of the polarization $\Lambda_m$ on $\operatorname{Hilb}_{4,1}^m$ is given by
the formula
$$
\varphi_m^* \Lambda_m
=
\left(\binom {m+1}{3}-\binom{m-2}{3}\right)\eta+\left(\binom{m}{3}-\binom{m-2}{3}\right)h,$$
where we use the convention that $\binom{a}{b}=0$ if $a<b$.
In particular, the slope of $\varphi^*_m\Lambda_m$ is given by
$$
t=\left\{
\begin{array}{ll}
\frac{m-2}{m+1} & 2\le m\le 4 \\
&\\
\frac{2m^2-8m+8}{3m^2-9m+8}& m\ge 5.\\
\end{array}
\right.
$$
\end{proposition}

\begin{proof}
This follows directly from the construction of $\varphi_m$ and is left to the reader.
\end{proof}

\subsubsection{The Chow variety}
The Hilbert-Chow morphism $\psi_\infty:\operatorname{Hilb}_{4,1}\to \operatorname{Chow}_{4,1}$
 induces a birational map
 $
\varphi_\infty: \mathbb PE \dashrightarrow \operatorname{Chow}_{4,1}
$.  We will denote by $\Lambda_\infty$ the canonical polarization on the Chow variety.   The following was established in the proof of  \cite[Thm.~2.11]{g4ball}.

\begin{proposition}[\cite{g4ball}]\label{maptochow}
The birational map
$$
\varphi_\infty:\mathbb PE \dashrightarrow \operatorname{Chow}_{4,1}
$$
restricts to a morphism on the locus of points $([q],[f])$ such that $q$ and $f$ do not have a common linear factor.  The pull-back of the canonical polarization $\Lambda_\infty$ on $\operatorname{Chow}_{4,1}$ is proportional to
$3\eta+2h$.\qed
\end{proposition}

\begin{remark}\label{remasymptotics}
There is a constant $c\in \mathbb Q_+$ such that $\lim_{m\to \infty}\frac{1}{3m^2}\psi_m^*\Lambda_m=c\psi_\infty^*\Lambda_\infty$ (cf. \cite[Thm.~4]{knudsenmumford}).  This is reflected in the slopes in Propositions \ref{comparepehilb} and \ref{maptochow}.
\end{remark}

\subsection{Cones of divisors  on $\mathbb PE$}  We now consider the nef cone and pseudoeffective cone of $\mathbb PE$.  Benoist \cite{benoist} has determined the nef cones of more general spaces of complete intersections.  We state a special case of  his result here, together with a basic observation on the pseudoeffective cone.

\begin{proposition}[{\cite[Thm 2.7]{benoist}}]\label{propnef}
The nef cone of $\mathbb PE$ has extremal rays of slope $0$ and $\frac{1}{2}$.  The pseudoeffective cone of $\mathbb PE$ has an extremal ray of slope $0$ and contains the ray of slope $\frac{34}{33}$.
 \end{proposition}

\begin{proof}
The computation of the nef cone is in \cite[Thm 2.7]{benoist}.   For the pseudoeffective cone, on the one hand, $\eta$ is effective (in fact semi-ample), but not big, so it generates one boundary of the pseudoeffective cone.   The discriminant divisor $\delta$ is effective, establishing the other claim.
\end{proof}

\subsection{The Rojas--Vainsencher resolution} \label{secrvres}
 Rojas--Vainsencher \cite{rv} have constructed an explicit resolution  $W$ of the rational map $\bP E \dashrightarrow \Hilb_{4,1}$, giving a diagram:
$$\xymatrix{
& W \ar[ld]_{\pi_1} \ar[rd]^{\pi_2} & \\
\bP E \ar@{-->}[rr] & & \Hilb_{4,1} .}$$
It is shown in  \cite[Thm.~3.1]{rv} that $W$ can be obtained from $\bP E$ via a sequence of seven blow-ups along $\SL(4)$-invariant smooth subvarieties, and the resulting space $W$ is isomorphic to $\mathbb PE$ along $U\subseteq \mathbb PE$, the locus of complete intersections.  In particular, $\SL(4)$ acts on $W$ (compatibly with the action on $\mathbb P E$), and $W$ is non-singular.

\section{Singularities of $(2,3)$-complete intersections}
In this section we discuss the possible isolated singularities of $(2,3)$-complete intersections in $\bP^3$.  Recall that given such a complete intersection, the quadric is uniquely determined by the curve, while the cubic is only determined modulo the quadric.  In the GIT analysis, the only relevant cases are when the quadric and cubic are not simultaneously singular, by which we mean that they have no common singular points.  In this case, we can choose either the quadric or cubic to obtain local coordinates and view the singularities of $C$ as planar singularities.

\subsection{Double Points}
The only planar singularities of multiplicity two are the $A_k$ singularities. We will see later in our GIT analysis that when $k$ is odd, it is important to distinguish between two types of $A_k$ singularities, those that separate the curve and those that do not.

\begin{proposition}
There exists a reduced $(2,3)$-complete intersection possessing a non-separating singularity of type $A_k$ if and only if $k \leq 8$.  Moreover, if $C$ is a $(2,3)$-complete intersection with a separating $A_k$ singularity at a smooth point of the quadric on which it lies, then one of the following holds:
\begin{enumerate}
\item  $k=9$, and $C$ is the union of two twisted cubics.
\item  $k=7$, and $C$ is the union of a quartic and a conic.
\item  $k=5$, and $C$ is the union of a quintic and a line.
\end{enumerate}
\end{proposition}

\begin{proof}
The local contribution of an $A_k$ singularity to the genus is $\lfloor \frac{k}{2} \rfloor$.  Since the arithmetic genus of a $(2,3)$-complete intersection is 4, it follows that it cannot admit an $A_k$ singularity if $k \geq 10$.  Conversely, it is easy to see that there exist $(2,3)$-complete intersections with non-separating singularities of type $A_k$ for each $k\le 8$ (e.g. see \cite[\S2.3.7]{maksymg4}).

If $C$ possesses a separating singularity of type $A_{2n-1}$, then $C = C_1 \cup C_2$, where $C_1$ and $C_2$ are connected curves meeting in a single point with multiplicity $n$.  A case by case analysis of the possibilities gives the second statement of the proposition.  It is straightforward to check that there is no $(2,3)$-complete intersection with a separating node or tacnode.
\end{proof}

\subsection{Triple Points}
\label{sect:triplepoints}

Let $C$ be a $(2,3)$-complete intersection with a singularity of multiplicity 3, which  does not contain a line component meeting the residual curve only at the singularity.  Notice that projection from the singularity maps $C$ onto a cubic in $\PP^2$.  It follows that $C$ is contained in the cone over this cubic.  We choose specific coordinates so that the singular point is $p = (1,0,0,0)$ and the tangent space to the quadric at $p$ is given by $x_3 = 0$.  Now, consider the 1-PS with weights $(1,0,0,-1)$.  The flat limit of $C$ under this one-parameter subgroup is cut out by the equations:
$$ x_0 x_3 + q'( x_1 , x_2 ) = f'( x_1 , x_2 ) = 0 $$
where $q'$ and $f'$ are forms in the variables $x_1 , x_2$.  We see that this limit is the union of three (not necessarily distinct) conics meeting at the points $p$ and $(0,0,0,1)$.

Following \cite{maksymg4} we will refer to these unions of conics as {\it tangent cones}.  In our GIT analysis we will see that, for a given linearization, the semistable tangent cones are precisely the polystable (i.e.~semi-stable with closed orbit) curves with triple point singularities.  Note that the conics are distinct if and only if the original triple point is of type $D_4$.

\subsection{Curves on Singular Quadrics}
As we vary the GIT parameters, we will see that certain subloci of curves on singular quadrics are progressively destabilized.  In this section we briefly describe each of these loci.  The first locus to be destabilized is the set of curves lying on low-rank quadrics.

\begin{proposition}
The only reduced $(2,3)$-complete intersections with more than one component of positive genus consist of two genus one curves meeting in 3 points.  Such a curve necessarily lies on a quadric of rank 2, and moreover the general complete intersection of a cubic and a rank 2 quadric is such a curve.
\end{proposition}

\begin{proof}
Suppose that $C = C_1 \cup C_2$ is the union of two positive genus curves.  Neither curve may have degree 2 or less, and hence both have degree 3.  Any degree 3 curve that spans $\bP^3$ is rational, and hence the two curves are both plane cubics.  Since $C$ is contained in a unique quadric, it follows that this quadric must be the union of two planes, and hence $C$ is as described above.
\end{proof}

Following \cite{afs}, we refer to such curves as {\it elliptic triboroughs}.  The locus of elliptic triboroughs is expected to be flipped in the Hassett--Keel program at the critical value $\alpha = \frac{5}{9}$.  This is exactly what we will prove in the following sections.

We now consider curves on a quadric of rank 3.  More specifically, we will see that a curve lies on a quadric cone if and only if its normalization admits a Gieseker--Petri special linear series. The proposition below follows by a standard argument.  The result is not needed in the ensuing proofs, but is useful in giving a geometric interpretation to the stability computations in later sections.

\begin{proposition} Let $C\subset \bP^3$ be a complete intersection of a cubic and a quadric of rank at least 3, non-singular everywhere except possibly one point.  Then the following hold:
\begin{enumerate}
\item  If $C$ is smooth, it has a vanishing theta-null if and only if it lies on a quadric cone.
\item  The normalization of $C$ is a hyperelliptic genus 3 curve if and only if $C$ lies on a quadric cone and has a node or cusp at the vertex.
\item  $C$ is a tacnodal curve such that the two preimage points of the tacnode via the normalization are conjugate under the hyperelliptic involution if and only if $C$ lies on a quadric cone and has a tacnode at the vertex.
\end{enumerate}\qed
\end{proposition}

\section{The two boundary cases}\label{sectboundary}
In this section we describe two previously studied birational models for $\overline{M}_4$ that are obtained via GIT for canonically embedded genus  $4$  curves (see \cite{maksymg4} and \cite{g4ball}).  In the later sections we will see that these two models coincide with the ``boundary cases'' in our GIT problem.  In other words, each of the models is isomorphic to a quotient of $\bP E \gquot \SL (4)$ for a certain choice of linearization, and all of the other linearizations we consider are effective combinations of these two.

\subsection{Chow Stability, following \cite{g4ball}}\label{cubic3folds}
Let $\Chow_{4,1}$ denote the irreducible component of the Chow variety containing genus 4 canonical curves.  In \cite{g4ball}, the authors study the GIT quotient $\Chow_{4,1} \gquot_{\Lambda_\infty} \SL (4)$ and obtain the following:

\begin{theorem}[{\cite[Thm.~3.1]{g4ball}}] \label{thmcubic}
The stability conditions for the quotient $\Chow_{4,1}\gquot_{\Lambda_\infty} \SL(4)$ are described as follows:
\begin{itemize}
\item[(0)] Every semi-stable point $c\in \Chow_{4,1}$ is the cycle associated to a $(2,3)$-complete intersection in $\mathbb P^3$. The only non-reduced $(2,3)$-complete intersections that give a semi-stable point $c\in \Chow_{4,1}$ are the genus $4$ ribbons (all with associated cycle equal to the twisted cubic with multiplicity $2$).
 \end{itemize}
Assume now  $C$ is a reduced  $(2,3)$-complete intersection in $\mathbb P^3$, with associated point $c\in \Chow_{4,1}$. Let $Q\subseteq \mathbb P^3$ be the unique quadric containing $C$. The following hold:
 \begin{itemize}
 \item[(0')] $c$ is unstable if $C$ is the intersection of a quadric and a cubic that are simultaneously singular. Thus, in items (1) and (2) below we can assume $C$ has only planar singularities.
 \item[(1)] $c$ is stable if and only if $\rank Q \ge 3$ and $C$ is a curve with  at worst $A_1,\ldots,A_4$ singularities at the smooth points of $Q$ and at worst an $A_1$ or $A_2$ singularity at the vertex of $Q$ (if $\rank Q=3$).
 \item[(2)] $c$ is strictly semi-stable if and only if
 \begin{itemize} \item[i)] $\rank Q=4$  and
\begin{itemize} \item[($\alpha$)]
$C$ contains a singularity of type $D_4$ or $A_5$, or,
\item[($\beta$)]
$C$ contains a singularity of type $A_k$, $k\ge 6$, and $C$ does not contain an irreducible component of degree $\leq 2$, or,
 \end{itemize}

  \item[ii)] $\rank Q=3$,   $C$ has at worst an $A_k$, $k\in \mathbb N$, singularity at the vertex of $Q$  and
  \begin{itemize} \item[($\alpha$)] $C$ contains a $D_4$ or an $A_5$ singularity at a smooth point of $Q$ or an $A_3$ singularity at the vertex of $Q$, or, \item[($\beta$)] $C$ contains a singularity of type $A_k$, $k\ge 6$, at a smooth point of $Q$ or a singularity of type $A_k$, $k\ge 4$, at the vertex of $Q$, and $C$ does not contain an irreducible component that is a line, or, \end{itemize}
  \item[iii)] $\rank Q=2$ and $C$ meets the singular locus of $Q$ in three distinct points.
\end{itemize}
\end{itemize}
\end{theorem}

\begin{remark}\label{remribbons}
In the example from \cite[\S 7]{BE}, it is shown that  up to change of coordinates there is only one canonically embedded ribbon of genus 4.  Moreover, it is shown that the ideal of this ribbon (again, up to change of coordinates) is generated by the quadric $q = x_1 x_3 - x_2^2$ and the cubic
$$f = \det \left( \begin{array}{ccc}
x_3 & x_2 & x_1 \\
x_2 & x_1 & x_0 \\
x_1 & x_0 & 0  \\
\end{array} \right).$$
\end{remark}

\begin{remark}\label{minorbit1}
The closed orbits of semi-stable curves fall into 3 categories (see also \cite[Rem.~3.2, 3.3]{g4ball}):
\begin{enumerate}
\item  The curve $C_D=V(x_0x_3,x_1^3+x_2^3)$, consisting of three pairs of lines meeting in two $D_4$ singularities;
\item  The maximally degenerate curve $C_{2A_5}=V(x_0x_3-x_1x_2,x_0x_2^2+x_1^2x_3)$ with two $A_5$ singularities;
\item  The curves $C_{A,B}=V(x_2^2-x_1x_3,Ax_1^3+Bx_0x_1x_2+x_0^2x_3)$, of which there is a pencil parameterized by $4A/B^2$.  If $4A/B^2\ne 0,1$, then $C_{A,B}$  has an $A_5$ singularity at a smooth point of the singular quadric, and an $A_3$ singularity at the vertex of the cone.  If $4A/B^2=0$, then $C_{A,B}$  has an $A_5$ and $A_1$ singularity at smooth points of the singular quadric, and an $A_3$ singularity at the vertex of the cone.   If $4A/B^2=1$ the curve $C_{A,B}$ is the genus $4$ ribbon, and the associated point in $\Chow_{4,1}$ is the twisted cubic with multiplicity $2$. Note also that the orbit closures of curves corresponding to cases (2) i) ($\beta$) and (2) ii) ($\beta$) contain the orbit of the ribbon.
\end{enumerate}
Moreover, we can describe the degenerations of the strictly semi-stable points $c\in\Chow_{4,1}$.  Let $C$ be a $(2,3)$-scheme with strictly semi-stable cycle $c\in \Chow_{4,1}$.  If $C$ contains a $D_4$ singularity, or lies on a rank $2$  quadric, then $c$ degenerates to the cycle associated to $C_D$.  If $C$ lies on a quadric $Q$ of rank at least $3$, and either $C$ contains an $A_5$ singularity at a smooth point of $Q$, or an $A_3$ singularity at the vertex of $Q$ (if $\rank Q=3$), then $c$ degenerates to either the cycle associated to  $C_{2A_5}$ or to the cycle associated to some $C_{A,B}$ with $4A/B^2\ne 1$.  Otherwise, $c$ degenerates to $C_{A,B}$ with $4A/B^2=1$,  a non-reduced complete intersection supported on a rational normal curve.
\end{remark}

Additionally, it is shown in \cite{g4ball} that the quotient of the Chow variety coincides with one of the Hassett--Keel spaces, specifically:
\begin{equation}\label{end1}
\Chow_{4,1} \gquot_{\Lambda_\infty} \SL(4)\cong \overline{M}_4 \left( \frac{5}{9} \right).
\end{equation}
For the reader's convenience, we briefly describe the birational contraction $\overline{M}_4 \dashrightarrow \Chow_{4,1} \gquot \SL(4)$ in Table \ref{table1}.  In order to make sense of the table, we need to recall some standard terminology.  Specifically, a \emph{tail} of genus $i$ is a genus $i$ connected component of a curve that meets the residual curve in one point.  Similarly, a \emph{bridge} of genus $i$ is a genus $i$ connected component of a curve that meets the residual curve in two points.  By conjugate points on a hyperelliptic curve, we mean points that are conjugate under the hyperelliptic involution.  An \emph{elliptic triborough} is a genus $1$ connected component of a curve that meets the residual curve in three points.

\begin{table}[htb!]
\begin{tabular}{|c|l|}
\hline
Semi-stable Singularity & Locus Removed in $\overline M_4$ \\
\hline\hline
$A_2$ & elliptic tails \\
\hline
$A_3$ & elliptic bridges \\
\hline
$A_4$ & genus 2 tails attached at a Weierstrass point \\
\hline
non-separating $A_5$ & genus 2 bridges attached at conjugate points \\
\hline
separating $A_5$ & general genus 2 tails \\
\hline
$A_6$ & hyperelliptic genus 3 tails attached at a Weierstrass \\
&point \\
\hline
non-separating $A_7$ & curves in $\Delta_0$ with hyperelliptic normalization glued \\
&at conjugate points \\
\hline
$A_8$, $A_9$, ribbons & hyperelliptic curves \\
\hline
$D_4$ & elliptic triboroughs \\
\hline
\end{tabular}
\vspace{0.2cm}
\caption{The birational contraction  $\overline M_4\dashrightarrow\Chow_{4,1}\gquot \SL(4)$}\label{table1}
\end{table}

\begin{remark} We note in particular that the rational map $\overline M_4\dashrightarrow\Chow_{4,1}\gquot \SL(4)$ contracts the boundary divisors $\Delta_1$ and $\Delta_2$, the closure of the hyperelliptic locus, and the locus of elliptic triboroughs.
\end{remark}

\subsection{Terminal Stability (i.e. stability for $(3,3)$ curves on quadric surfaces) following  \cite{maksymg4}}\label{sect33}
Recall that every canonically embedded curve $C$ of genus 4 is contained in a quadric in $\bP^3$.  If this quadric is smooth, then it is isomorphic to $\bP^1 \times \bP^1$, and $C$ is a member of the class $\vert \mathcal{O}_{\bP^1 \times \bP^1} (3,3) \vert$.  The automorphism group of the quadric is $\SO(4)$, which is isogenous to $\SL(2) \times \SL(2)$.  The GIT quotient $ \vert \mathcal{O} (3,3) \vert \gquot \SO(4) $ was studied in detail by Fedorchuk in \cite{maksymg4}.  Because this GIT quotient appears as the last stage of the log minimal model program for $\overline{M}_4$, we refer to curves that are (semi)stable with respect to this action as {\bf terminally (semi)stable}.  We summarize the results of \cite{maksymg4} here.

\begin{theorem} [{Fedorchuk \cite[\S2.2]{maksymg4}}]\label{thmmaksym}
Let $C \in \vert \OO (3,3) \vert$.  $C$ is terminally stable if and only if its has at worst double points as singularities and  it does not contain a line component $L$ meeting the residual  curve $C'=\overline{C\setminus L}$ in exactly one point.  $C$ is terminally semi-stable if and only if it contains neither a double-line component, nor a line component $L$ meeting the residual curve $C'$ in exactly one point, which is also a singular point of $C'$ (i.e.~$L\cap C'=\{p\}$ and  $p\in \Sing(C')$).
\end{theorem}

\begin{remark}\label{minorbit2}
The closed orbits of strictly semi-stable curves fall into  $4$  categories:
\begin{enumerate}
\item  The maximally degenerate curve $C_{2A_5}=V(x_0x_3-x_1x_2,x_0x_2^2+x_1^2x_3)$ with 2 $A_5$ singularities (same curve as in Rem. \ref{minorbit1}(2));
\item  The triple conic $V(x_0x_3 - x_1 x_2 ,x_3^3)$;
\item  Unions of a smooth conic and a double conic meeting transversally.  As discussed in Remark 2.4 in \cite{maksymg4}, there is a one-dimensional family of such curves;
\item  Unions of three conics meeting in two $D_4$ singularities $V(x_0x_3-x_1x_2,x_1^3+x_2^3)$ (analogue of the case of Rem. \ref{minorbit1}(1)).
\end{enumerate}
\end{remark}

As mentioned above, Fedorchuk \cite{maksymg4} showed that this GIT quotient is the final non-trivial step in the Hassett--Keel program for genus $4$, specifically:
\begin{equation}\label{end2}
 \vert \mathcal{O} (3,3) \vert \gquot \SO(4) \cong \overline{M}_4\left[\frac{29}{60},\frac{8}{17}\right)\to \overline{M}_4\left(\frac{8}{17}\right)=\{\ast\}.
\end{equation}
In this paper we are interested in describing the behavior of the Hassett--Keel program for genus $4$ curves in the interval $\alpha\in \left[\frac{8}{17}, \frac{5}{9}\right]$ (with endpoints described by \eqref{end2} and \eqref{end1} respectively). In particular, in the following sections, we will give an explicit factorization of  the birational map
$$ \Psi: \overline{M}_4\left(\frac{5}{9}\right)\cong \Chow_{4,1} \gquot \SL (4) \dashrightarrow \vert \mathcal{O} (3,3) \vert \gquot \SO(4)\cong \overline{M}_4\left[\frac{29}{60},\frac{8}{17}\right) $$
as the composition of two flips and a divisorial contraction.

For the moment, by comparing the stability conditions given by Theorems \ref{thmcubic} and \ref{thmmaksym} and by simple geometric considerations, we obtain a rough description of the birational map $\Psi$ as summarized in  Table \ref{table2} (see also \cite[Table 1]{maksymg4}).
The first three lines of the table correspond to strictly semi-stable points of $\Chow_{4,1}$ that are all flipped by the map $\overline{M}_4 ( \frac{5}{9} - \epsilon ) \to \overline{M}_4 ( \frac{5}{9} )$.  Then, note that every Chow-stable curve contained in a quadric cone is terminally unstable.  There are three types of such curves: those that do not meet the vertex of the cone, those that meet it in a node, and those that meet it in a cusp.  These correspond to the latter three lines in the table, as well as the three critical slopes in our VGIT problem.  These last three lines correspond, in order, to the flip at $\alpha = \frac{23}{44}$, the flip at $\alpha = \frac{1}{2}$, and the divisorial contraction at $\alpha = \frac{29}{60}$.
\begin{table}[htb!]
\begin{tabular}{|c|l|}
\hline
Semi-stable Singularity & Locus Removed \\
\hline \hline
non-separating $A_5$ & tacnodal curves glued at conjugate points \\
\hline
$A_6$, non-sep. $A_7$, $A_8$, $A_9$  & ribbons (see Rem. \ref{minorbit1}(3)) \\
\hline
$D_4$ & elliptic triboroughs \\
\hline
separating $A_7$ & cuspidal curves with hyperelliptic normalization \\
\hline
contains a double conic & nodal curves with hyperelliptic normalization \\
\hline
triple conic & curves with vanishing theta-null \\
\hline
\end{tabular}
\vspace{0.2cm}
\caption{The birational map  $\Chow_{4,1} \gquot \SL (4) \dashrightarrow \vert \mathcal{O} (3,3) \vert \gquot \SO(4) $ }\label{table2}
\end{table}

\section{Numerical stability of Points in $\bP E$}\label{sectstability}
In this section we determine the stability conditions on $\bP E$ as the slope $t$ of the linearization varies by using the Hilbert--Mumford numerical criterion.
We note that a discussion of the Hilbert--Mumford index in a related and more general situation than ours was done by Benoist \cite{benoist}, whose results we are using here.

A technical issue arises in this section. Namely, we are interested in applying the numerical criterion for slopes $t\in\left(0,\frac{2}{3}\right]$. However, by Proposition \ref{propnef}, the linearizations of slope $t\ge \frac{1}{2}$ are not ample.  Thus, for $t\ge \frac{1}{2}$, special care is needed to define a GIT quotient $\bP E\gquot_t\SL(4)$ and to understand the stability conditions by means of the numerical criterion. In this section we make the necessary modifications to handle this non-standard GIT case. Namely, here we work with ``numerical'' (semi-)stability instead of the usual (Mumford) (semi-)stability. Then, in Section \ref{S:Non-AmpleBundles}, we prove that there is no difference between the two notions of stability and that everything has the expected behavior. In short, for slopes $t\in\left(0,\frac{1}{2}\right)$ everything works as usual, since the linearization is ample. For $t\ge \frac{1}{2}$ one can still proceed as in the ample case, but this is justified only \emph{a posteriori} by the results of Section \ref{S:Non-AmpleBundles}.

\subsection{The numerical criterion for $\bP E$}
Let us start by recalling the Hilbert--Mumford index for hypersurfaces.  That is, we consider the case of $\SL(r+1)$ acting on $\mathbb PH^0(\mathbb P^r,\mathscr O_{\mathbb P^r}(d))$.
In this case, given a one-parameter subgroup (1-PS) $\lambda:\mathbb G_m\to \SL(r+1)$, the action on $H^0(\mathbb P^r,\mathscr O_{\mathbb P^r}(1))$ can be diagonalized.  We describe the action of $\lambda$ in these coordinates  with a weight vector $\alpha = (\alpha_0 , \alpha_1 , \ldots , \alpha_r)$.  For a monomial $x^a = x_0^{a_0} \cdots x_r^{a_r}\in H^0(\mathbb P^r,\mathscr O_{\mathbb P^r}(d))$ in these coordinates, we define the $\lambda$-weight of $x^a$ to be
$$ wt_{\lambda} (x^a) = \alpha.a = \alpha_0 a_0 + \alpha_1 a_1 + \ldots + \alpha_r a_r. $$
The Hilbert--Mumford invariant associated to a non-zero  homogeneous  form $F\in H^0(\mathbb P^r,\mathscr O_{\mathbb P^r}(d))$ and a 1-PS $\lambda$  is then given by
$$
\mu(F,\lambda)=\max_{x^a \text{ monomials in } F}wt_\lambda(x^a).
$$

Following \cite{benoist}, the Hilbert--Mumford index for complete intersections $V(f,q)$ has a simple expression in terms of the indices  for the associated hypersurfaces.

\begin{proposition}[{\cite[Prop 2.15]{benoist}}]\label{computemu}
The Hilbert--Mumford index of a point $([q],[f]) \in \bP E$ is given by
$$\mu^{a \eta + bh} (([q],[f]), \lambda ) = a \mu (q, \lambda ) + b \mu (f, \lambda ) ,$$
where $f \in H^0 ( \bP^3 , \mathcal{O} (3))$ is a representative of $[f]$ of minimal $\lambda$-weight.
\end{proposition}

Recall from \S \ref{Preliminaries} that the {\it slope} of the line bundle $a\eta +bh$ is defined to be $t=\frac{b}{a}$.  Throughout we will write $\mu^t (([q],[f]), \lambda )$ for the Hilbert--Mumford index with respect to the linearization $\eta+th$.

\begin{definition} We say that $([q],[f])$ is {\it numerically $t$-stable} (resp.~{\it numerically $t$-semi-stable}) if, for all non-trivial one-parameter subgroups $\lambda$,
$$ \mu^t (([q],[f]), \lambda ) > 0 \text{ (resp. } \ge 0 \text{ )} .$$
\end{definition}

While we will typically only refer to numerical (semi-)stability for points of $\mathbb PE$, we will occasionally want to refer to this notion in more generality.  Recall that the definition can be made in the situation where one has a reductive group $G$ acting on a proper space $X$ with respect to a linearization $L$ (\cite[Def.~2.1, p.48]{GIT}).  We will use the notation $X^{nss}$ and $X^{ns}$ to refer to the numerically semi-stable, and numerically stable loci respectively.

\begin{remark}\label{remnumstab}
We recall that for the general GIT set-up, with a reductive group $G$ acting on a space $X$ with respect to a linearization $L$, Mumford \cite[Def. 1.7]{GIT} defines  a point $x\in X$ to be {\it semi-stable} (and a similar definition for stable) if there exists an invariant section $\sigma\in H^0(X,L^{\otimes n})$ such that $\sigma(x)\neq 0$ and $X_{\sigma}$ is affine. We will use the standard notation $X^{(s)s}$ to denote the (semi-)stable points in this sense. To emphasize the distinction with numerical (semi-)stability, and avoid confusion, we will sometimes refer to this as Mumford (semi-)stability.  For \emph{ample} line bundles on projective varieties, the Hilbert--Mumford numerical criterion (\cite[Thm. 2.1]{GIT}) gives that numerical (semi-)stability agrees with  (semi-)stability.  If $L$ is not ample, however, the notions  may differ (see e.g.~Remark \ref{remexample}).  In our situation, we work with numerical stability, since it is easily computable; in the end (using the results in Section \ref{S:Non-AmpleBundles}), we will prove that this is same as Mumford stability. Of course, this distinction is only relevant in the non-ample case (i.e.~linearizations of slopes $t\ge \frac{1}{2}$).
\end{remark}

\begin{remark}\label{remexample}
The following simple example illustrates some of the differences between numerical stability and Mumford stability.  Let $G$ be a reductive group acting on a smooth projective variety $X$ with $\dim(X)\ge 2$, and let  $L$ be an ample linearization.   Consider the blow-up $\pi: X'\to X$ along a closed $G$-invariant locus $Z$ (with $\mathrm{codim} Z\ge 2$) that contains at least one semi-stable point $p\in Z\cap X^{ss}$.
 Note that the rings of invariant sections $R(X,L)^G$ and $R(X',\pi^*L)^G$ agree via pullback of sections, and the Hilbert-Mumford indices agree by functoriality (\cite[iii), p.49]{GIT}).   It follows that any point $q$ in the fiber $\pi^{-1}(p)$ (contained in the exceptional divisor $E$) will be numerically  semi-stable. But no such point  can be Mumford semi-stable, because the pull-back of a section $\sigma$ that does not vanish at $q$ does not vanish on $\pi^{-1}(p)$, and consequently $X'_\sigma$ can not be  affine.
\end{remark}

\begin{notation}\label{GITnotation}
When considering GIT quotients,  we will use the notation $X\gquot_LG$ for the categorical quotient of the semi-stable locus $X^{ss}$ (\cite[Thm.~1.10]{GIT}); we will call this the (categorical)  GIT quotient.  Note this may not necessarily agree with $\operatorname{Proj}R(X,L)^G$ when $L$ is not ample.
\end{notation}

\subsection{Application of the numerical criterion}\label{sectnumerical}
We begin our discussion by identifying points of $\bP E$ that fail to be numerically semi-stable for any linearization.
Note that in order to show that a certain pair $([q],[f])$ is not $t$-numerically semi-stable, it suffices to find a $1$-PS $\lambda$ and a representative $f$ such that $\mu(q,\lambda)+t\mu(f,\lambda)<0$, since for any representative $f$, one has $\mu^t(([q],[f]),\lambda)\le \mu(q,\lambda)+t\mu(f,\lambda)$ (cf. Proposition \ref{computemu}).

\begin{proposition}
\label{NoReducibleQuadric}
If $q$ is a reducible quadric, then $(q,f)$ is not numerically $t$-semi-stable for any $t < \frac{2}{3}$.  Moreover, if $q$ and $f$ share the common linear factor $x_0$, then $(q,f)$ is destabilized by the 1-parameter subgroup with weights $(-3,1,1,1)$ for any $t \leq \frac{2}{3}$.
\end{proposition}

\begin{proof}
Suppose that $q$ is singular along the line $x_2 = x_3 = 0$, and consider the 1-PS $\lambda$ with weights $(1,1, -1, -1)$.  Then $\mu (q, \lambda ) = -2$ and $\mu (f, \lambda ) \leq 3$.  Hence $ \mu^t ((q,f), \lambda ) < 0$, so $(q,f)$ is not numerically $t$-semi-stable.

To see the second statement, let $\lambda$ be the 1-PS with weights $(-3,1,1,1)$ and note that $\mu (q, \lambda ) \leq -2$, $\mu (f, \lambda ) \leq -1$.
\end{proof}

Note that, as a consequence, every numerically $t$-semi-stable point of $\bP E$ for $t < \frac{2}{3}$ is a complete intersection.  The only points of $\bP E$ that do not correspond to complete intersections are those where $q$ and $f$ share a common linear factor.  Henceforth, we will refer interchangeably to stability of the point $(q,f) \in \bP E$ and stability of the curve $C = V(q,f)$.

\begin{proposition}
If $q$ and $f$ are simultaneously singular, then $(q,f)$ is not numerically $t$-semi-stable for any $t \le \frac{2}{3}$.
\end{proposition}

\begin{proof}
Suppose that $q$ and $f$ are both singular at the point $(1,0,0,0)$, and consider the 1-PS with weights $(3,-1,-1,-1)$.  Then $\mu (q, \lambda ) = -2$ and $\mu (f, \lambda ) \leq 1$.  Hence $ \mu^t ((q,f), \lambda ) \leq - \frac{4}{3} < 0$, so $(q,f)$ is not numerically $t$-semi-stable.
\end{proof}

\begin{proposition}
\label{ConePoint}
Suppose that $q$ is a quadric cone and $f$ passes through the singular point $p$ of $q$.  If $p$ is not a node or a cusp of $C$, then $(q,f)$ is not numerically $t$-semi-stable for any $t < \frac{2}{3}$.
\end{proposition}

\begin{proof}
Without loss of generality, we may assume that $q = x_1 x_3 - x_2^2$.  We write the cubic in coordinates as
$$ f = \sum_{a+b+c+d=3} \alpha_{a,b,c,d} x_0^a x_1^b x_2^c x_3^d .$$
If $p = (1,0,0,0)$ is not a node of $C$, then the projectivized tangent cone to $C$ at $p$ is a double line contained in the quadric cone.  Hence, the tangent space to the cubic at $p$ meets the quadric in a double line.  We may therefore assume that this tangent space is the plane $x_3 = 0$.  It follows that $\alpha_{3,0,0,0} = \alpha_{2,1,0,0} = \alpha_{2,0,1,0} = 0$.  Since $p$ is not a cusp, we have $\alpha_{1,2,0,0} = 0$ as well.  Now, consider the 1-PS with weights $(3,1,-1,-3)$.  Then $\mu (q , \lambda ) = -2$ and $\mu (f, \lambda ) \leq 3$.  It follows that $ \mu^t ((q,f), \lambda ) < 0$, so $(q,f)$ is not numerically $t$-semi-stable.
\end{proof}

\begin{corollary}\label{corribbon}
Ribbons are not numerically $t$-semi-stable for any $t <\frac{2}{3}$.
\end{corollary}

\begin{proof}  This follows from Remark \ref{remribbons} and  the proposition above.
\end{proof}

\begin{proposition}
\label{LineComponent}
Suppose that $C$ contains a line $L$ and let $C' =\overline{C \backslash L}$ be the residual curve.  If $p \in L\cap C'$ is a singular point of $C'$, then $C$ is not numerically $t$-stable for any $t \le \frac{2}{3}$.  If, in addition, $L$ meets $C'$ with multiplicity $3$ at $p$, then $C$ is not numerically $t$-semi-stable for any $t \le \frac{2}{3}$.  In particular, if $C$ contains a double line, then it is not numerically $t$-semi-stable for any $t \le \frac{2}{3}$.
\end{proposition}

\begin{proof}
By Propositions \ref{NoReducibleQuadric} and \ref{ConePoint}, we may assume that the singular point $p = (1,0,0,0)$ is a smooth point of the quadric $q$.  Without loss of generality, we may assume that the line $L$ is cut out by $x_2 = x_3 = 0$ and that the tangent plane to $q$ at $p$ is cut out by $x_3 = 0$.  As above, we write the cubic in coordinates as
$$ f = \sum_{a+b+c+d=3} \alpha_{a,b,c,d} x_0^a x_1^b x_2^c x_3^d .$$
By replacing $f$ with a cubic of the form $f - ( \alpha x_0 + \beta x_1 )q$ for suitable choices of $\alpha$ and $\beta$, we obtain a representative for $f$ such that $\alpha_{2,0,0,1} = \alpha_{1,1,0,1} = 0$.  From the assumption that $C$ contains $L$, we may conclude that $\alpha_{a,3-a,0,0} = 0$ ($a=0,1,2,3$).  From the assumption that $C'$ is singular at $p$, we may further conclude that $\alpha_{2,0,1,0} = \alpha_{1,1,1,0}  = 0$.
Now consider the 1-PS $\lambda$ with weights $(1,0,0,-1)$.  Then $\mu (q, \lambda ) \leq 0$ and $\mu (f, \lambda ) \leq 0$.  It follows that $\mu^t ((q,f), \lambda ) \le 0$, so $(q,f)$ is not numerically $t$-stable.

Let us now assume further that $L$ meets $C'$ with multiplicity 3 at $p$.  Then we obtain in addition that $\alpha_{0,2,1,0} = 0$.  Considering the 1-PS $\lambda$ with weights $(3,1,-1,-3)$, we see that $\mu (q, \lambda ) \leq 0$ and $\mu (f, \lambda ) \leq -1$.  It follows that $\mu^t ((q,f), \lambda ) < 0$, so $(q,f)$ is not numerically $t$-semi-stable.

The case of a double line follows by taking the reduced line and its residual curve; i.e.~$C=2L+C''=L+C'$ where $C'=L+C''$.
\end{proof}

\begin{proposition}
\label{TriplePoints}
If $C$ has a singularity of multiplicity greater than two, it is not numerically $t$-stable for any $t \leq \frac{2}{3}$.  Moreover, if $C$ has a singularity of multiplicity greater than three, it is not numerically $t$-semi-stable for any $t \leq \frac{2}{3}$.
\end{proposition}

\begin{proof}
Without loss of generality, we may assume that the singular point is $p = (1,0,0,0)$ and by Proposition \ref{LineComponent} we may assume that $C$ does not contain a line $L$ through $p$.

Let us first consider the case where $p$ is a triple point.  Because $C$ does not contain any lines $L$ such that $L \cap \overline{C \backslash L} = \{p\}$, projection from $p$ maps $C$ onto a cubic in $\PP^2$.  Hence, $C$ is contained in the cone over this cubic.  Consequently, this cone gives a representative $f$ for $[f]$, which we will fix for the computations that follow.
Suppose now that the tangent space to the quadric at $p$ is given by $x_3 = 0$.  Then consider the 1-PS with weights $(1,0,0,-1)$.  We see that both $\mu (q, \lambda ) \leq 0$ and $\mu (f, \lambda ) \leq 0$ and hence $C$ is not numerically $t$-stable for any $t$.

Now let us consider the case where $p$ has multiplicity $4$.
Projection from $p$ maps $C$ onto a conic in $\PP^2$.  Since $C$ is contained in the cone over this conic, it follows that $p$ is the singular point of a quadric cone containing $C$.  We have already seen, however, that unless $p$ is a node or cusp of $C$, then $C$ is not numerically $t$-semi-stable.
\end{proof}

We now consider three curves that are terminally semi-stable, but not Chow semi-stable.  We determine those values of $t$ at which they become numerically unstable.

\begin{proposition}
\label{A7Sing}
If $C$ contains a conic $C'$ that meets $\overline{C \backslash C'}$ in an $A_7$ singularity, it is numerically $t$-unstable for all $t > \frac{6}{11}$.  If $q$ is a quadric cone and $C$ has a cusp at the singular point of $q$, it is numerically $t$-unstable for all $t < \frac{6}{11}$.
\end{proposition}

\begin{proof}
First, consider the case where $C$ contains a conic $C'$ meeting the residual curve in an $A_7$ singularity.  Without loss of generality, we assume that the conic is contained in the plane $x_3 = 0$, the singularity occurs at the point $(1,0,0,0)$, and the quadric $\frac{f}{x_3}$ contains the line $x_2 = x_3 = 0$.  By assumption, the tangent space to $q$ at this point contains this line, and the quadric $\frac{f}{x_3}$ is singular.  Now, consider the 1-PS with weights $(7,3,-1,-9)$.  Then $\mu (q , \lambda ) \leq 6$ and $\mu (f, \lambda ) \leq -11$.  It follows that
$$ \mu^t ((q,f), \lambda ) \leq -11t+6 $$
which is negative when $t > \frac{6}{11}$.

Now, consider the case where $q$ is a quadric cone and $C$ has a cusp at the singular point of $q$.  Without loss of generality, we may assume that $q = x_1 x_3 - x_2^2$.  We write a representative for the cubic in coordinates as
$$ f = \sum_{a+b+c+d=3} \alpha_{a,b,c,d} x_0^a x_1^b x_2^c x_3^d .$$
As above, we may assume that the tangent space to the cubic at the cone point of $q$ is the plane $x_3 = 0$.  It follows that $\alpha_{3,0,0,0} = \alpha_{2,1,0,0} = \alpha_{2,0,1,0} = 0$.  Consider the 1-PS with weights $(9,1,-3,-7)$.  Then $\mu (q , \lambda ) = -6$ and $\mu (f , \lambda ) \leq 11$.  It follows that
$$ \mu^t (([q],[f]), \lambda ) \leq 11t-6 $$
which is negative when $t < \frac{6}{11}$.
\end{proof}

\begin{remark}
We will see in Theorem \ref{mainthm2} (3) that the minimal orbit of the above strictly semi-stable curves at $t=\frac{6}{11}$ is given by
$$x_1^2+x_0x_2 = x_2^2x_3+x_0x_3^2 = 0 .$$
This curve consists of two components meeting in a separating $A_7$ singularity.  One of the components is a conic.  The other is a quartic with a cusp at the vertex of the cone.
\end{remark}

\begin{proposition}
\label{DoubleConic}
If $C$ contains a double conic component, it is numerically $t$-unstable for all $t > \frac{2}{5}$.  If $q$ is a quadric cone and $f$ passes through the singular point of the cone, then it is is numerically $t$-unstable for all $t < \frac{2}{5}$.
\end{proposition}

\begin{proof}
First, consider the case where $C$ contains a double conic component.  Without loss of generality, we may assume that the conic is contained in the plane $x_0 = 0$.  Consider the 1-PS with weights $(-3,1,1,1)$.  Then $\mu (q , \lambda ) \leq 2$, and, since $f$ is divisible by $x_0^2$, we have $\mu (f, \lambda ) \leq -5$.  It follows that
$$ \mu^t ((q,f), \lambda ) \leq -5t+2 $$
which is negative when $t > \frac{2}{5}$.

Now, consider the case where $q$ is a quadric cone and $f$ passes through the singular point.  Without loss of generality, we may assume that the singularity occurs at the point $p=(1,0,0,0)$.  Consider the 1-PS with weights $(3,-1,-1,-1)$.  Then, since $q$ is singular at $p$, $\mu (q, \lambda ) = -2$.  Furthermore, since $f$ contains $p$, $\mu (f, \lambda ) \leq 5$, so
$$ \mu^t ((q,f), \lambda ) \leq 5t-2 $$
which is negative when $t < \frac{2}{5}$.
\end{proof}

\begin{remark}
We will see in Theorem \ref{mainthm2} (4) that here, the relevant minimal orbit of strictly semi-stable curves is given by the union of two rulings of a quadric cone and a double conic:
$$ q(x_0,x_1,x_2) = x_0x_3^2 = 0 .$$
\end{remark}

\begin{proposition}
\label{TripleConic}
If $C$ is a triple conic, then it is numerically $t$-unstable for all $t > \frac{2}{9}$.  If $q$ is singular, then $(q,f)$ is numerically $t$-unstable for all $t < \frac{2}{9}$.
\end{proposition}

\begin{proof}
First, consider the case where $C$ is a triple conic.  Without loss of generality, we may assume that $f = x_0^3$.  Consider the 1-PS with weights $(-3,1,1,1)$.  Then $\mu (q, \lambda ) \leq 2$ and $\mu (f, \lambda ) = -9$.  Hence
$$ \mu^t ((q,f), \lambda ) \leq -9t+2 $$
which is negative when $t > \frac{2}{9}$.

Now, consider the case where $q$ is singular.  Without loss of generality, we may assume that the singular point is the point $p=(1,0,0,0)$.  Consider the 1-PS with weights $(3,-1,-1,-1)$.  Then, since $q$ is singular at $p$, we have $\mu (q, \lambda ) = -2$ and $\mu (f, \lambda ) \leq 9$, so
$$ \mu^t ((q,f), \lambda ) \leq 9t-2 $$
which is negative when $t < \frac{2}{9}$.
\end{proof}

We now change directions, and establish stability in some cases.  First, we recall a basic result from GIT.

\begin{lemma}\label{lemGIT}
Let $X$ be a scheme (of finite type over an algebraically closed field $k$) and let $G$ be a reductive algebraic group (over $k$) acting on $X$.  Suppose $L$ is a $G$-linearized line bundle on $X$.   There is a natural  induced action of $G$ on $X^{ss}$ and an induced linearization on $L|_{X^{ss}}$ so that there is an isomorphism of categorical quotients $
X/\!\!/_L G\cong X^{ss}/\!\!/_{L|_{X^{ss}}}G$.
Moreover, if $X$ is complete, $L$ is ample, and  $X^{ss}\ne \emptyset$, then  there exists $m,n_0\in \mathbb N$ such that for all $n\ge n_0$, $
H^0(X^{ss}, (L_{ss}^{\otimes m})^{\otimes n})^G=H^0(X,(L^{\otimes m})^{\otimes n})^G$.
\end{lemma}

\begin{proof}  First we consider the GIT quotients $X/\!\!/_L G$ and $X^{ss}/\!\!/_{L|_{X^{ss}}}G$.  If $X^{ss}=\emptyset$, then the statement of the lemma is vacuous, so we may assume $X^{ss}\ne \emptyset$.
Then the  injective restriction maps
$
H^0(X,L^{\otimes n})^G\to H^0(X^{ss},L|_{X^{ss}}^{\otimes n})^G
$
make it clear that any $x\in X^{ss}$ is semi-stable for the $G$-linearization of $L|_{X^{ss}}$.  Thus the semi-stable loci agree.  Since the $G$-action on $X^{ss}$ is induced from that on $X$, one concludes there is an isomorphism  $X/\!\!/_L G\cong X^{ss}/\!\!/_{L|_{X^{ss}}}G$ of the  categorical quotients.

Now let us consider the spaces of global sections
$$H^0(X^{ss}, (L_{ss}^{\otimes m})^{\otimes n})^G \ \ \text{ and } \ \ H^0(X,(L^{\otimes m})^{\otimes n})^G.$$
We are now assuming that $X$ is complete, $L$ is ample and $X^{ss}\ne \emptyset$.  First, note that there is a surjection
$
\pi:X^{ss}\to  X\gquot_L G$.
Moreover, we have a line bundle $\mathscr O(1)$ on $X\gquot_L G$ such that (up to rescaling $L$) we have $\pi^*\mathscr O(1)=L_{ss}$.  By construction of $\operatorname{Proj}$, and using the assumption that $X$ is complete and $L$ is ample, so that $X\gquot_L G=\operatorname{Proj}R(X,L)^G$, we get $H^0 ( X\gquot_L G, \mathscr O(n))=H^0( X,  L^{\otimes n})^G$
 (for $n\gg 0$).    Finally, by construction,
 $
 H^0(X^{ss},L_{ss}^{\otimes n})^G=H^0(X\gquot_L G, \mathscr O(n))$  completing the proof.
\end{proof}

We use this lemma in the following.

\begin{lemma}
\label{TerminalStability}
If $0<t < \frac{2}{9}$, then $(q,f)$ is $t$-(semi)stable if and only if $q$ is smooth and $f \vert_q$ is terminally (semi)stable.
\end{lemma}

\begin{proof}
Note that, when $0 < t < \frac{2}{9}$, the line bundle $\eta + th$ is ample, so in this case numerical (semi-)stability is the same as actual (semi-)stability.  Let $Q$ be the smooth quadric defined by $x_0^2 + x_1^2 + x_2^2 + x_3^2 = 0$ and write $i: \bP E_Q \hookrightarrow \bP E$ for the inclusion of the fiber of $\bP E$ over $Q$.  Write $G = \SL(4)$ and $G' = \SO(4)$ for the stabilizer of $Q$.  Consider the quasi-projective variety $\widetilde{\bP E_Q} = G \times_{G'} \bP E_Q$, which is the quotient of $G \times \bP E_Q$ by the free action of $G'$:  $h(g,x) = (gh^{-1} , hx)$ for $h \in G'$.  There is a natural identification of the ring of invariants (cf.~\cite[p.10 Eq.~(3)]{kirwannr}):
\begin{equation}\label{invring}
R':=\bigoplus_{n\ge0} H^0(\bP E_Q , n i^* (\eta + th))^{G'} \cong \bigoplus_{n\ge0} H^0(\widetilde{\bP E_Q }, n(\eta +th))^{G}.
\end{equation}
Notice that $\bP E_Q$ has Picard rank 1, so $i^* (\eta +th) = \mathcal{O} (d)$ for some $d \geq 0$.

Now, observe that $\widetilde{ \bP E_Q }$ is isomorphic to the open set $V \subset \bP E$ parameterizing pairs $(q,f)$ where $q$ is smooth.  To see this, note that $G \times \bP E_Q$ admits a $G'$-invariant map to this space sending $(g,f)$ to $(g \cdot Q , g \cdot f)$.  This map induces an isomorphism on the quotient because the quadric $q$ is uniquely determined by an element of $G/G'$.

Finally, note that when $t < \frac{2}{9}$, every numerically $t$-semi-stable point lies on a smooth quadric.  From the computations above it follows that $\mathbb PE^{ss}\subseteq V$.  Thus, by virtue of  Lemma \ref{lemGIT},
$$ H^0 ( \bP E , n(\eta + th))^G \cong H^0 ( V , n(\eta + th))^G $$
for these values of $t$.
Hence
$$ \bP E \gquot_t G = \Proj \bigoplus_{n \geq 0} H^0 ( \bP E , n(\eta + th) )^G = \Proj \bigoplus_{n \geq 0} H^0 ( V , n(\eta + th) )^G $$
$$ = \Proj \bigoplus_{n \geq 0} H^0 ( \bP E_Q , \mathcal{O} (n) )^{G'} = \bP E_Q \gquot G' .$$

\end{proof}

\section{Quotients of the Hilbert Scheme}\label{chowhilb}
 A standard approach to constructing birational models of $\overline{M}_g$ is to consider the pluricanonical image of a curve as a point in a Chow variety or Hilbert scheme.  One can then construct the GIT quotient of this Chow variety or Hilbert scheme by the group of automorphisms of the ambient projective space.  This approach can be found, for example, in both Mumford's and Gieseker's constructions of $\overline{M}_g$ as an irreducible projective variety (see \cite{mumford}, \cite{Gieseker}).  It is also the method by which Schubert \cite{Schubert} constructed the moduli space of pseudostable curves $\overline{M}_g^{ps}$, and   Hassett and Hyeon \cite{hh2}  constructed the first flip in the Hassett--Keel program.  In our situation, we will consider the GIT quotients $\operatorname{Hilb}_{4,1}^m\gquot_{\Lambda_m}\operatorname{SL}(4)$.  Recall that points of $\operatorname{Hilb}_{4,1}^m$ are called $m$-th Hilbert points.

\subsection{Numerical criterion for finite Hilbert stability}
A criterion for stability of Hilbert points was worked out in \cite{hhl}.  We briefly review their results.

Let $X \subset \PP^N$ be a variety with Hilbert polynomial $P(m)$.  We will write $k_m = \binom{N+m}{m} - P(m)$. For any $v \in \mathbb{R}^{N+1}$, we define an ordering $<_v$ on the set of monomials in $N+1$ variables as follows:

$x^a <_v x^b$ if
\begin{enumerate}
\item  $\deg x^a < \deg x^b$;
\item  $\deg x^a = \deg x^b$ and $v.a < v.b$;
\item  $\deg x^a = \deg x^b$, $v.a = v.b$, and $x^a <_{Lex} x^b$ in the lexicographic order.
\end{enumerate}
In particular, given a 1-PS $\lambda$ with weight vector $\alpha = (\alpha_0 , \alpha_1 , \ldots , \alpha_N )$, the monomial order $<_{\lambda}$ is the lexicographic order associated to the weight $\alpha$.  For each polynomial $f$, let $in_{<_{\lambda}} (f)$ denote the largest term of $f$ with respect to $<_{\lambda}$.  For an ideal $I$, we define $in_{<_{\lambda}} (I) = \langle in_{<_{\lambda}} (f) \vert f \in I \rangle$.

\begin{proposition}[\cite{hhl}]
\label{NumericalCriterion}
A point $I \in \Hilb^m_{4,1} \subset \mathbb G(k_m,n_m)$ is semi-stable if and only if, for every 1-PS $\lambda$, we have
$$ \sum_{x^a \in in_{<_{\lambda}} (I)} wt_{\lambda} (x^a) \geq 0 $$
where the left-hand sum is over the monomials $x^a$ of degree $m$ in $in_{<_{\lambda}} (I)$.
\end{proposition}

Note that when $k_m = 1$, this criterion coincides with the criterion for hypersurfaces described in \S \ref{sectstability}.

\begin{proposition}
\label{HilbertCompleteIntersection}
If $I \in \Hilb^m_{4,1}$ is not the $m^{th}$ Hilbert point of a $(2,3)$-complete intersection, then it is not $m$-Hilbert semi-stable for any $m \geq 2$.  Similarly, if $X \in \Chow_{4,1}$ is not a complete intersection, then it is not Chow-semi-stable.
\end{proposition}

\begin{proof}
Let $I \in \Hilb^m_{4,1} \subset \mathbb{G} (k_m , n_m )$ be a vector space.  We note that there is a quadric $q$ and a cubic $f$, not divisible by $q$, such that $I$ contains all monomials of the form $q x^a$ and $f x^b$, where $x^a$ is a monomial of degree $m-2$ and $x^b$ is a monomial of degree $m-3$.  Indeed, this condition is closed in $\mathbb{G} (k_m , n_m )$, so it is satisfied by every element of $\Hilb^m_{4,1}$.  If $q$ and $f$ do not share a common linear factor, then $I$ is necessarily the $m^{th}$ Hilbert point of the intersection $q=f=0$.

Assume that $q$ and $f$ share a common linear factor.  We may choose coordinates such that $q = x_0 x_1$, and $f$ is divisible by $x_0$.  We may further assume that $f$ has a nonzero $x_0 x_3^2$ term.  Now, consider the 1-PS $\lambda$ with weights $(-3,1,1,1)$.  By definition, $I$ contains all of the monomials of the form $x_0 x_1 x^a$, where $x^a$ is a monomial of degree $m-2$, and of the form $x_0 x_3^2 x^b$, where $x^b$ is a monomial of degree $m-3$.  The number of such monomials is
$$ \binom{m+1}{3} + \binom{m}{3} - \binom{m-1}{3} = \frac{1}{6} (m-1)(m^2 + 4m-6) $$
and the total weight of these monomials is
$$ -2 \binom{m+1}{3} - \binom{m}{3} = - \frac{1}{2} m^2 (m-1). $$
It follows that
$$ \sum_{x^a \in in_{<_{\lambda}} (I)} wt_{\lambda} (x^a)$$
$$ \leq m \left[ \binom{m+3}{3} - (6m-3) - \frac{1}{6} (m-1) (m^2 + 4m-6) \right] - \frac{1}{2} m^2 (m-1) $$
$$ = \frac{1}{2} m(m-2)(m-3) - \frac{1}{2} m^2 (m-1) = -m(2m-3). $$
Since this is negative for all $m \geq 2$, we see that $I$ is not $m$-Hilbert semi-stable for these same $m$.  We obtain the analogous result for the Chow variety by noting that $\lim_{m \to \infty} \frac{-m(2m-3)}{m^2} < 0.$
\end{proof}

We would like to compare the numerical criterion for points in the Hilbert scheme to the numerical criterion for points on $\bP E$.  To this end, we have the following:

\begin{proposition}
\label{NumericalComparison}
Suppose $y \in \Hilb_{4,1}$ corresponds to a $(2,3)$-complete intersection $C\subseteq \mathbb P^3$.  Denote also by $y$ the corresponding point in $\mathbb PE$.  There exists a positive constant $c\in \mathbb Q$ such that
for any 1-parameter subgroup $\lambda$, we have $\mu^{\frac{2}{3}} (y, \lambda ) \geq c\mu^{\psi_\infty^*\Lambda_{\infty}} (y, \lambda )$.
\end{proposition}

\begin{proof}
This follows directly from
Benoist  \cite[Prop. 4.3]{benoist} and Remark \ref{remasymptotics}.
\end{proof}

\section{Quotients of the Rojas--Vainsencher resolution}\label{S:Non-AmpleBundles}
In this section we complete the arguments needed in Section \ref{sectstability} (esp. for Theorem \ref{mainthm2}) to handle  GIT for non-ample bundles on $\bP E$.  The main point is to use the results on Hilbert stability of the previous section together with the Rojas--Vainsencher resolution $W$ of the rational map $\bP E\dashrightarrow \Hilb_{4,1}$ (see \S \ref{secrvres}):
$$\xymatrix{
& W \ar[ld]_{\pi_1} \ar[rd]^{\pi_2} & \\
\bP E \ar@{-->}[rr] & & \Hilb_{4,1} }.$$

\subsection{Study of GIT stability on $W$}   Generally speaking, the key to understanding GIT quotients for non-ample bundles is to relate them to quotients of birational models with (semi)ample linearizations. In our situation, we consider the birational model $W$ of $\bP E$ with linearizations of the form
$$\alpha \pi_1^* \eta + \beta \pi_2^*(\psi_\infty^* \Lambda_{\infty}).$$ Note that for $\alpha, \beta \geq 0$, these linearizations  are semiample on $W$.
\begin{notation}
  Set $L(t):=\eta + th$ (on $\bP E$) and $\Lambda:=c\psi_\infty^*\Lambda_\infty$ (on $\Hilb_{4,1}$), where $c$ is the constant in Proposition \ref{NumericalComparison}.      Let $M(t)= \alpha \pi_1^*L(0)+ \beta \pi_2^*\Lambda$ (on $W$), where $\alpha$ and $\beta$ are such that $L(t)= \alpha L(0)+ \beta L(\frac{2}{3})$ (N.B. $\rank \Pic(\bP E)=2$).  We will write $W^{ss} (t)$ for the semistable locus on $W$ with respect to the linearization $M(t)$, and $\bP E^{nss} (t)$ for the numerically semistable locus with respect to $L(t)$.
\end{notation}

We start by making the following observations on the behavior of GIT on $W$.
\begin{proposition}
\label{StabilityAgrees}
$W^{ss} (t) \subseteq \pi_1^{-1} ( \bP E^{nss} (t))$.
\end{proposition}

\begin{proof}
First, suppose that $y \in W$ is in the exceptional locus of the map $\pi_1$.  Then $\pi_1 (y)$ lies in the locus of pairs $(q,f)$ such that $q$ and $f$ share a common linear factor.  Similarly, $\pi_2 (y)$ is not a complete intersection of a quadric and a cubic.  It follows from Proposition \ref{NoReducibleQuadric} that, for the 1-PS $\lambda$ with weights $(-3,1,1,1)$, we have
\begin{eqnarray*}
 \mu^{\pi_1^* L(0)} (y, \lambda ) &<& 0 \\
 \mu^{\pi_1^* L(\frac{2}{3})} (y, \lambda ) &<& 0.
\end{eqnarray*}
Moreover, it follows from Proposition \ref{HilbertCompleteIntersection} that
$$ \mu^{\pi_2^* \Lambda} (y, \lambda ) < 0. $$
By the linearity of the Hilbert--Mumford index, $y$ is numerically unstable for all the line bundles in question.  It follows that $W^{nss} (t)$ is contained in the ample locus of $M(t)$, and thus $W^{nss} (t) = W^{ss} (t)$.

Now suppose that $y \notin \pi_1^{-1} ( \bP E^{nss} (t))$ is not in the exceptional locus of the map $\pi_1$.  By Proposition \ref{NumericalComparison} together with the linearity and functoriality of the Hilbert--Mumford index, there is a one-parameter subgroup $\lambda$ such that
$$0 > \mu^{\alpha L(0) + \beta L(\frac{2}{3})} (y, \lambda ) \geq \mu^{\alpha \pi_1^* L(0) + \beta \pi_2^* \Lambda} (y, \lambda ) .$$
It follows that $y \notin W^{ss} (t)$.
\end{proof}

A consequence of Proposition \ref{StabilityAgrees} is that, for every $t$ in the range $0 < t \leq \frac{2}{3}$, $W^{ss} (t)=W^{nss}(t)$ is contained in the locus on which $\pi_2$ restricts to an isomorphism.  It follows that every invariant section of $M(t)$ has affine non-vanishing locus, hence the usual results about GIT hold for the linearization $M(t)$ despite the fact that it is only semi-ample, rather than ample.  As another consequence, we may think of points in $W^{ss} (t)$ as $(2,3)$-complete intersections.  Combining Proposition \ref{StabilityAgrees} with the results of \S \ref{sectnumerical}, we can identify many $t$-unstable points in $W$.  It remains to show that each curve that has not been explicitly destabilized thus far is in fact $t$-semi-stable.  We will prove this in Theorem \ref{mainthm2}.   This type of argument is related in spirit to the potential stability argument used by Gieseker and Mumford for the GIT construction of $\overline M_g$ (e.g.~see \cite[\S4.C]{hm}).

Finally, we recall briefly the notion of the basin of attraction from \cite[Def.~4]{hl}.  If the stabilizer of a curve $C'$ contains a 1-PS $\lambda$, then the \emph{basin of attraction} (of $C'$ with respect to $\lambda$)  is defined to be
$$ A_{\lambda} (C') := \{ C \ \vert\  C \text{ specializes to } C' \text{ under } \lambda \} .$$
If $C'$ is strictly semi-stable with respect to $\lambda$, meaning that $\mu (C' , \lambda ) = 0$, then  $C'$ is semi-stable if and only if $C$ is semi-stable for every (equivalently, any) $C \in A_{\lambda} (C')$ (see \cite[Lem.~4.3]{hh2}).

We are now ready to prove the following key result describing the stability on the space $W$ which interpolates between $\bP E$ and $\Hilb_{4,1}$. The main advantages  here are: (1) on $W$ we are in a standard GIT set-up (i.e.~(semi-)ample linearizations, as opposed to the situation on  $\bP E\gquot_t\SL(4)$ for $t>\frac{1}{2}$), and (2) the natural spaces $\Hilb^m_{4,1}\gquot \SL(4)$ are then easily described using $W\gquot_t\SL(4)$.

\begin{theorem}\label{mainthm2}
Let $C \in W^{ss} (t)$.  Then $C$ is a complete intersection of a quadric and a cubic in $\bP^3$, and:
\begin{enumerate}
\item  $C \in W^{(s)s} ( \frac{2}{3} )$ if and only if it is Chow (semi-)stable.
\item \label{teomain22} $C \in W^{(s)s} (t)$ for all $t \in ( \frac{6}{11} , \frac{2}{3} )$ if and only if it is Chow (semi-)stable, but not a ribbon, an elliptic triborough, or a curve on a quadric cone with a tacnode at the vertex of the cone.  The closed orbits of strictly $t$-semi-stable points correspond to the maximally degenerate curve with $A_5$ singularities (i.e. $C_{2A_5}$ in the notation of \S \ref{sectboundary}) and the unions of three conics meeting in two $D_4$ singularities (see Remark \ref{minorbit2} (1) and (4)).
\item  $C \in W^{ss} (t)$ for all $t \in ( \frac{2}{5} , \frac{6}{11} )$ if and only if
    \begin{enumerate}
    \item  $C \in W^{ss} (t)$ for $t\in (\frac{6}{11},\frac{2}{3})$ and is not an irreducible cuspidal curve with hyperelliptic normalization, or
    \item  $C$ contains a conic that meets the residual curve in a singularity of type $A_7$, but otherwise satisfies condition (2) of Theorem \ref{thmcubic}.
    \end{enumerate}
    The closed orbits of strictly $t$-semi-stable points are the same as for $t \in ( \frac{6}{11} , \frac{2}{3} )$.
\item  $C \in W^{ss} (t)$ for $t \in ( \frac{2}{9} , \frac{2}{5} )$ if and only if
    \begin{enumerate}
    \item  $C \in W^{ss} (t)$ for $t \in ( \frac{2}{5} , \frac{6}{11} )$ and is not an irreducible nodal curve with hyperelliptic normalization, or
    \item  $C$ has a triple-point singularity whose tangent cone is the union of a double conic and a conic meeting in two points, but otherwise satisfies condition (2) of Theorem \ref{thmcubic}.
    \end{enumerate}
    The closed orbits of strictly $t$-semi-stable points correspond to the maximally degenerate curve with $A_5$ singularities, the unions of three conics meeting in two $D_4$ singularities, and the unions of a conic and a double conic meeting at two points (see Remark \ref{minorbit2} (1), (3) and (4)).
\item  $C \in W^{(s)s} (t)$ for $t \in (0, \frac{2}{9} )$ if and only if it is contained in a smooth quadric and it is terminally (semi-)stable.
\end{enumerate}
\end{theorem}
\begin{proof}
As already mentioned, the strategy of the proof is to show that every curve that has not been explicitly destabilized by using  the results of \S \ref{sectnumerical}  and Proposition \ref{StabilityAgrees} is in fact $t$-semi-stable. We start by proving items (1) and (5), which  identify the quotients corresponding to the two end chambers of the VGIT on $W$ with the two GIT quotients discussed in \S \ref{sectboundary}. We then identify GIT walls using $t$-semi-stable curves with positive dimensional stabilizer. We use the basin of attraction to determine $t$-semi-stable curves at each wall.  By general variation of GIT, each such curve that is contained in a smooth quadric is in fact $t$-semi-stable for all smaller values of $t$.  In this way we identify the majority of $t$-semi-stable curves.  To establish the $t$-semi-stability of the remaining curves, we use another basin of attraction argument.

\vspace{0.25cm}

\noindent{\it Proof of (1)}.  The isomorphism $\mathbb PE\gquot_{\frac{2}{3}}\SL(4)\cong \Chow_{4,1}\gquot \SL(4)$ was established in   \cite{g4ball}. Since $M(\frac{2}{3})=\Lambda$ is semi-ample (it is the pull-back of the natural polarization $\mathscr O(1)$ on $\Chow_{4,1}$), one obtains the identification  $\operatorname{Proj}R\left(W,M(\frac{2}{3})\right)^{\SL(4)}=\operatorname{Proj}R(\Chow_{4,1},\mathscr O(1))^{\SL(4)}$.    In fact, although $\Lambda$ is only semi-ample, we have shown that  $W^{ss}(M(t))=W^{nss}(M(t))$,
and so one may also conclude that the (categorical) GIT quotients agree: $W\gquot_{M(\frac{2}{3})}\SL(4)\cong \Chow_{4,1}\gquot \SL(4)$.  (QED (1))

\vspace{0.25cm}

\noindent{\it Proof of (5)}.  Suppose now that $0<t < \frac{2}{9}$ (in particular, $L(t)$ is ample on $\bP E$) and note that by Prop.~\ref{StabilityAgrees} and Lem.~\ref{TerminalStability}, $W^{ss} (t)\subseteq \pi_1^{-1}(\mathbb PE^{nss}(t))=\pi_1^{-1}(\mathbb PE^{ss}(t))$ is contained in the open set $V$ consisting of pairs $(q,f)$ where $q$ is smooth.  Since $\pi_1$ restricts to an isomorphism on the open set $V$ and $M(t) \vert_V = L(t) \vert_V$,
(5)  follows from Lemmas \ref{lemGIT} and  \ref{TerminalStability}. (QED  (5))

\vspace{0.25cm}

We now turn to the intermediate chambers.  By general variation of GIT, we know that if $C \in W^{ss} ( \epsilon ) \cap W^{ss} ( \frac{2}{3} )$, then $C \in W^{ss} (t)$ for all $t$ in the range $\epsilon < t < \frac{2}{3}$.  On the other hand, suppose that $C$ is neither Chow semi-stable nor terminally semi-stable.  It follows that one of the following must be true:
\begin{enumerate}
\item  $C$ contains a line $L$ such that $L \cap \overline{C \backslash L}$ is a singular point of the residual curve;
\item  $C$ has a singularity of multiplicity greater than three;
\item  $C$ is contained in a quadric of rank 1 or 2;
\item  $C$ is contained in a quadric cone and has a singularity of type other than $A_2$, $A_3$ or $A_4$ at the singular point of the cone;
\item  $C$ is contained in a quadric cone \textit{and} has a separating $A_7$ singularity, or
\item  $C$ is contained in a quadric cone \textit{and} has a triple-point singularity of type other than $D_4$.
\end{enumerate}
By Proposition \ref{StabilityAgrees} and the results of \S \ref{sectnumerical}, any of the first four possibilities imply that $C \notin W^{ss} (t)$ for any $t < \frac{2}{3}$.  The fifth case can only be $t$-semi-stable for $t \in [ \frac{2}{9} , \frac{6}{11} ]$.  In the last case, $C$ specializes to its ``tangent cone'' under the one-parameter subgroup described in \S \ref{sect:triplepoints}.  By  the proof of Proposition \ref{TriplePoints}, this one-parameter subgroup has weight zero on $C$, and hence if $C$ is $t$-semi-stable then its tangent cone is $t$-semi-stable as well.  Since the singularity is not of type $D_4$, the tangent cone is non-reduced.  It cannot be a triple conic unless $t= \frac{2}{9}$, because by Proposition \ref{TripleConic} a triple conic and a curve on a singular quadric cannot be simultaneously semi-stable except at this critical value.  It therefore must be the union of a conic and a double conic on a quadric cone, which can only be $t$-semi-stable for $t \in [ \frac{2}{9} , \frac{2}{5} ]$.

\vspace{0.25cm}
Having destabilized the necessary curves, we now turn our attention to showing that various curves are (semi-)stable for particular values of $t$.

\vspace{0.25cm}

\noindent{\it Proof of (2)}. We consider first the $t$-interval $( \frac{6}{11} , \frac{2}{3} )$.  By the above, every $t$-semi-stable point for $t \in ( \frac{6}{11} , \frac{2}{3} )$ is either terminally semi-stable or Chow semi-stable.  The only
terminally polystable curves that are not Chow semi-stable are the triple conic, the double conics, and the curves with separating $A_7$ singularities (\S \ref{cubic3folds}, \S\ref{sect33}), and none of these can be $t$-semi-stable for $t > \frac{6}{11}$ (\S \ref{sectnumerical}).  It follows that
$W^{ss} (t) \subset W^{ss} ( \frac{2}{3} )$ for all $t$ in this interval.  As a consequence, since a wall $t_0$ of this GIT chamber is characterized by $W^{ss} ( t_0 ) \nsubseteq W^{ss} ( \frac{2}{3} )$, the wall must lie outside the open $t$-interval $( \frac{6}{11} , \frac{2}{3} )$.  By general variation of GIT, we therefore have that $W^s ( \frac{2}{3} ) \subset W^{ss} (t)$ for $t$ in this
interval, and $W^{ss} ( \frac{2}{3} ) \cap W^{ss} (\epsilon) \subset W^{ss} (t)$ for all $0 < \epsilon < t$ (and in particular for $0 < \epsilon < \frac{2}{9}$) as well.

Thus it remains to determine the $t$-semi-stability of the remaining strictly Chow semi-stable points that are not terminally semi-stable.  These all lie on the quadric cone.  Considering the possibilities from Theorem \ref{thmcubic}, we see that the only such curves that have not already been destabilized are the curves on the quadric cone with $A_k$  ($k\ge 5$) singularities  (that do not have an $A_n$  ($n\ge 3$) singularity  at the vertex of the cone) and the curves on the quadric cone with $D_4$ singularities.

Suppose first that $C = V(q,f)$ is a $\frac{2}{3}$-semi-stable curve on the quadric cone that has an $A_k$  ($k\ge 5$) singularity at a smooth point of the cone, but does not have an  $A_n$ ($n\ge 3$) singularity  at the vertex of the cone.    We argue by contradiction that $C$ is also $t$-semi-stable.  Suppose that $\lambda$ is a one-parameter subgroup such that $\mu^t ((q,f), \lambda ) < 0$.  By standard facts from variation of GIT, one can assume  that $\mu^{\frac{2}{3}} ((q,f), \lambda ) = 0$  (see e.g.~\cite[\S 4.1.2 ]{lazasurvey}).   Now let $C'=V(q',f')$ be the specialization of $C$ under $\lambda$.   Since  $\lambda$ fixes $C'$, it follows from the basin of attraction argument  that $C'$ is $\frac{2}{3}$-semi-stable as well.  The only $\frac{2}{3}$-semi-stable curve in the orbit closure of $C$, however, is a curve of the form
$$C'=C_{A,B}=V(x_2^2-x_1x_3,Ax_1^3+Bx_0x_1x_2+x_0^2x_3),$$  whose stabilizer in the given coordinates is the $\mathbb{C}^*$ with weights $\pm (3,1,-1,-3)$.  All of the curves that specialize to $C_{A,B}$ under the $1$-PS with weights $(3,1,-1,-3)$ have an $A_n$ ($n\ge 3$) singularity at the vertex of the cone.   Consequently,  $\lambda$ must be the $1$-PS with weights $(-3,-1,1,3)$.   This gives  $\mu(q', \lambda ) = 2$ and $\mu(f' , \lambda ) = -3$, so that  $\mu^t ((q',f'), \lambda ) > 0$ (any  other representative of $f'$ will have weight $\ge -3$).  Now since $(q',f')$ is the limit of $(q,f)$ under $\lambda$, we have $\mu^t ((q,f), \lambda ) =\mu^t ((q',f'), \lambda ) > 0$, a contradiction.

Similarly, if $C$ has a $D_4$ singularity, then $C$ specializes to its tangent cone under the one-parameter subgroup described in \S \ref{sect:triplepoints}, and this one-parameter subgroup has weight zero on $C$ by Proposition \ref{TriplePoints}.  Hence $C$ is $t$-semi-stable if and only if its tangent cone is $t$-semi-stable as well.  Now, suppose that there is a one-parameter subgroup $\lambda$ such that $\mu^t ((q,f), \lambda ) < 0$.  As in the previous case, we see that $\lambda$ must be contained in the stabilizer of the $\frac{2}{3}$-polystable limit of $C$, which is $C_D=V(x_0x_3,x_1^3+x_2^3)$.  The stabilizer of $C_D$ is the 2-dimensional torus consisting of one-parameter subgroups with weights of the form $\pm (a,-1,-1,2-a)$.  Since $C$ specializes to $C_D$ under $\lambda$ and $C$ is not contained in a reducible quadric, we see that $\lambda$ has weights of the form $(a,-1,-1,2-a)$.  But then $\mu(x_0x_3, \lambda ) = 2$ and $\mu(x_1^3+x_2^3 , \lambda ) = -3$, so as above,  $\mu^t ((q,f), \lambda ) > 0$, a contradiction.  (QED (2))

\vskip  0.25cm
The proofs of the remaining parts are similar.  We include the details for the convenience of the reader.

\vspace{0.25cm}

\noindent{\it Proof of (3)}. We next consider the $t$-interval $( \frac{2}{5} , \frac{6}{11} )$.  For $t < \frac{6}{11}$, cuspidal curves with hyperelliptic normalization can no longer be $t$-semi-stable, so there must be a GIT wall at $t = \frac{6}{11}$.  This implies that there is a $\frac{6}{11}$-semi-stable curve with positive dimensional stabilizer that is not $t$-semi-stable for $t = \frac{6}{11} + \epsilon$.  Reviewing the possibilities, we see that there is only one possible such curve, namely
$$ C' = V(x_1^2+x_0x_2 , x_2^2x_3+x_0x_3^2),$$
which has both a separating $A_7$ singularity and a cusp at the vertex of the quadric cone on which it lies.  If $C \in W^{ss} (\frac{6}{11} + \epsilon )$ is not in $W^{ss} ( \frac{6}{11} - \epsilon )$, then the orbit of $C$ under the $\mathbb{C}^*$ that stabilizes $C'$ must contain $C'$ in its closure.  It follows that, up to change of coordinates, $C$ must be of the form:
$$ C = V(x_1^2 + x_0x_2 + \alpha x_0^2 + \beta x_0x_1, x_2^2x_3 + x_0x_3^2 + f(x_0 , x_1 , x_2)),$$
where $\alpha , \beta$ are constants and $f$ is a cubic.  In other words, $C$ must be contained in a singular quadric and have a cusp at the vertex.  We therefore see that every $(\frac{6}{11} + \epsilon )$-semi-stable curve that is not of this form is $(\frac{6}{11} - \epsilon )$-semi-stable as well.

To identify the remaining $\frac{6}{11}$-semi-stable curves, we use the basin of attraction of $C'$.  Namely, since the curve $C'$ is $\frac{6}{11}$-semi-stable, we see that every curve in the basin of attraction of $C'$ is also $\frac{6}{11}$-semistable.  By Proposition \ref{A7Sing}, we see that this includes every curve with a separating $A_7$ singularity apart from those that we have explicitly destabilized already.  If such a curve $C$ is contained in a smooth quadric, then $C \in W^{ss} ( \frac{6}{11} ) \cap W^{ss} (0)$, so $C \in W^{ss} (t)$ for all $t \in [ 0, \frac{6}{11} ]$.

It remains to show that the curves contained in a quadric cone with a separating $A_7$ singularity are in fact $( \frac{6}{11} - \epsilon )$-semi-stable.  So let $C=V(q,f)$ be such a curve.  To show $C$ is $t$-semi-stable, we argue as above, noting that if $\lambda$ is a 1-PS such that $\mu^t ((q,f), \lambda ) < 0$, then $\lambda$ must be contained in the stabilizer of the $\frac{6}{11}$-polystable limit of this curve, which is the curve $C'$ above.  The stabilizer of $C'$ is a one-dimensional torus, so this determines the 1-PS $\lambda$ uniquely.  Indeed, in these coordinates, $\lambda$ must be the 1-PS with weights $(7,3,-1,-9)$.  Then $\mu (x_1^2+x_0x_2 , \lambda ) = 6$ and $\mu (x_2^2x_3+x_0x_3^2, \lambda ) = -11$, so as above  $\mu^t ((q,f), \lambda ) > 0$, a contradiction.

To complete this part of the proof, we note that by the above we obtain the inclusion $W^{ss} (t) \subset W^{ss} ( \frac{6}{11} - \epsilon )$ for all $t \in ( \frac{2}{5} , \frac{6}{11} )$, and hence this interval is contained in a single GIT chamber.  (QED (3))

\vspace{0.25cm}

\noindent{\it Proof of (4)}.  By arguments nearly identical to the previous case, we identify a GIT wall at $t = \frac{2}{5}$ corresponding to the curve $C' = V( x_0 x_2 - x_1^2 , x_1x_3^2)$,
which is the union of a double conic and two rulings of a quadric cone.  As before, if $C \in W^{ss} (\frac{2}{5} + \epsilon )$ is not in $W^{ss} ( \frac{2}{5} - \epsilon )$, then the orbit of $C$ under the stabilizer of $C'$ must contain $C'$ in its closure.  It follows that, up to change of coordinates, $C$ must be of the form $C = V( q(x_0,x_1,x_2), f(x_0,x_1,x_2,x_3))$,
where $f$ is a cubic containing the vertex $(0,0,0,1)$.  In other words, $C$ must be contained in a singular quadric and have a node at the vertex.  We therefore see that every $(\frac{2}{5} + \epsilon )$-semi-stable curve that is not of this form is $(\frac{2}{5} - \epsilon )$-semi-stable as well.

As in the previous case, we see that every curve with a double conic component, apart from those we have explicitly destabilized, is $\frac{2}{5}$-semi-stable, as such curves are in the basin of attraction of $C'$.  Specifically, if a curve $C$ contains a double conic component that is contained in the plane $x_0 = 0$, then $C$ specializes to $C'$ under the 1-PS with weights $(-3,1,1,1)$, which is contained in the stabilizer of $C'$.   Furthermore, if such a curve is contained in a smooth quadric then it is contained in $W^{ss} ( \frac{2}{5} ) \cap W^{ss} (0)$, and hence it is $t$-semi-stable for all $t \leq \frac{2}{5}$.

It remains to show that the double conics contained in a quadric cone are $( \frac{2}{5} - \epsilon )$-semi-stable as well.  For this, we argue as above, noting that if $\lambda$ is a 1-PS such that $\mu^t ((q,f), \lambda ) < 0$ for such a curve $C=V(q,f)$, then $\lambda$ must be contained in the stabilizer of the $\frac{2}{5}$-polystable limit of this curve, which is the curve $C'$ above.  The stabilizer of $C'$ is the two-dimensional torus consisting of one parameter subgroups with weights $\pm (a,1,2-a,-3)$.  All the curves that specialize to $C$ under a 1-PS with weights $(-a,-1,a-2,3)$ pass through the vertex of the cone, so $\lambda$ must have weights $(a,1,2-a,-3)$.  But then $\mu (x_0 x_2 - x_1^2 , \lambda ) = 2$ and $\mu (x_1x_3^2, \lambda ) = -5$, so as above,  $\mu^t ((q,f), \lambda ) > 0$, a contradiction.  The fact that the entire $t$-interval $( \frac{2}{9} , \frac{2}{5} )$ is contained in a GIT chamber follows exactly as above. (QED (4)). \end{proof}

\begin{remark}
Note that in the theorem, points that are strictly semi-stable on a wall may become stable in the adjacent chamber.  For instance, for $t=\frac{2}{3}$, the ribbon is semi-stable, and the strictly semi-stable points corresponding to curves with $A_8,A_9$ singularities degenerate to this curve.  For $\frac{6}{11} <t< \frac{2}{3}$, the ribbon is unstable, but the curves with $A_8,A_9$ singularities become stable (not just semi-stable).
\end{remark}

\begin{remark}
The argument above also determines semi-stability conditions at the GIT walls.
\begin{enumerate}
\item  At $t= \frac{6}{11}$, both irreducible cuspidal curves with hyperelliptic normalization and curves with a separating $A_7$ singularity are strictly semi-stable.  The orbit closure of either type of curve contains the point
    $$x_1^2+x_0x_2 = x_2^2x_3+x_0x_3^2 = 0 .$$
\item  At $t= \frac{2}{5}$, both irreducible nodal curves with hyperelliptic normalization and double conics are strictly semi-stable.  The orbit closure of either type of curve contains the union of a double conic and two rulings on the quadric cone, given by
    $$ q(x_0,x_1,x_2) = x_0x_3^2 = 0 .$$
\item  At $t= \frac{2}{9}$, both curves contained in a quadric cone and triple conics are strictly semi-stable.  The orbit closure of either type of curve contains the triple conic on  a quadric cone.
\end{enumerate}
\end{remark}

\subsection{Comparing the GIT quotients}
 We set
$$
W\gquot_t \SL(4):=W\gquot_{M(t)} \SL(4)=\operatorname{Proj}R(W,M(t))^{\SL(4)},
$$
where recall $W\gquot_{M(t)} \SL(4)$ is the categorical quotient of the semi-stable locus, and the equality on the right holds because $W^{nss} (t)$ is contained in the ample locus of $M(t)$.

As discussed in Section \ref{sectstability}, the GIT quotient $ \bP E \gquot_t \SL(4)$ makes sense as a categorical quotient for all $t$. However, for non-ample linearizations (i.e. $t\ge \frac{1}{2}$), it is not \emph{a priori} clear how to describe it in terms of the numerically (semi-)stable  points (e.g.~Rem. \ref{remexample}).  Here we note that Proposition \ref{StabilityAgrees} and Theorem \ref{mainthm2} allows us to interpret our numerical results from the previous section as honest GIT results on the resolution $W$, and then the expected properties of $ \bP E \gquot_t \SL(4)$ follow (as well as the connection between numerical stability and stability).

\begin{corollary}\label{corWPE}
For $t \in [0, \frac{2}{3}]$, $ \bP E \gquot_{L(t)} \SL(4) = W \gquot_{M(t)} \SL(4)$ and for both spaces, numerical (semi-)stablility agrees with Mumford (semi-)stability.    Moreover, the ring of invariant sections $R(\mathbb PE,L(t))^{\SL(4)}$ is finitely generated and $$\mathbb PE\gquot_{L(t)} \SL(4)=\operatorname{Proj}R(\mathbb PE,L(t))^{\SL(4)}.$$
\end{corollary}

\begin{proof}
The boundary cases $t=0$ and $t=\frac{2}{3}$ have been proven already.  For $t\in (0, \frac{2}{3} )$, $W^{ss} (M(t)) \subseteq \pi_1^{-1} ( \bP E^{nss} (L(t)))$ by Prop \ref{StabilityAgrees}.  On the other hand, in Theorem \ref{mainthm2} we showed that every curve that is not explicitly destabilized in \S \ref{sectnumerical} is in fact semi-stable in $W$, so $\pi_1^{-1} ( \bP E^{nss} (L(t))) \subseteq W^{ss} (M(t))$. By Prop.~\ref{NoReducibleQuadric}, we see that $\pi_1^{-1} ( \bP E^{nss} (L(t)))$ is contained in the locus where $\pi_1$ restricts to an isomorphism identifying $\bP E^{nss} (L(t))$ and $\pi_1^{-1} ( \bP E^{nss} (L(t)))$.  Thus the categorical quotient of $ \bP E^{nss} (L(t))$ agrees with the categorical quotient $W\gquot_{M(t)} \SL(4)$, which equals $\operatorname{Proj}R(M(t))^{\SL(4)}$.

Now consider the injective restriction maps:
$$
H^0(W,M(t))^{\SL(4)}\to H^0(W^{ss},M(t)|_{W^{ss}})^{\SL(4)}
$$
$$
H^0(\mathbb PE,L(t))^{\SL(4)}\to H^0(\mathbb PE^{nss},L(t)|_{\mathbb PE^{nss}})^{\SL(4)} .
$$
The map on the top is in fact surjective (up to possibly taking a higher tensor power of $M(t)$) by Lemma \ref{lemGIT}.  The map on the bottom is surjective as well.  This follows for $t \leq \frac{2}{9}$ by Lemma \ref{lemGIT}, and for $\frac{2}{9} < t \leq \frac{2}{3}$  since the complement of $\mathbb PE^{nss}$ has codimension at least two.  Since $\mathbb PE^{nss}$ is identified with $W^{ss}$, and $M(t) \vert_{W^{ss}} \cong L(t) \vert_{\mathbb PE^{nss}}$, we get the equality we need.

It follows immediately that $R(\mathbb PE,L(t))^{\SL(4)}$ is finitely generated, and gives the same projective variety as $R(W,M(t))^{\SL(4)}$.  It is also elementary to check from this equality of invariant sections, that  Mumford stability and numerical stability then agree on $\mathbb PE$, since this holds on $W$.    Thus we have $$\mathbb PE\gquot_{L(t)} \SL(4)=W\gquot_{M(t)}\SL(4) =\operatorname{Proj}R(M(t))^{\SL(4)}=\operatorname{Proj}R(L(t))^{\SL(4)}.$$
\end{proof}

We now compare the GIT quotients of $W$ to those of the Hilbert scheme.

\begin{theorem}
We have the following isomorphisms of GIT quotients:
\begin{enumerate}
\item  $\Chow_{4,1} \gquot \SL(4) \cong W \gquot_{M(\frac{2}{3})} \SL(4) .$
\item  $\Hilb^{m}_{4,1} \gquot \SL(4) \cong W \gquot_{M(t)} \SL(4)$, where $t=\frac{m-2}{m+1}$ for $2\le m \le 4$ and $t = \frac{2(m-2)^2}{3m^2 -9m+8}$ $\forall m \geq 5$.
\end{enumerate}
\end{theorem}

\begin{proof} (1) was established in the proof of Theorem \ref{mainthm2}.  (2)
Let $U \subset \mathbb PE$ be the open set parameterizing complete intersections (see Rem. \ref{remdefu}) and $U_m \subset \Hilb^m_{4,1}$ be the corresponding open subset of $\Hilb^m_{4,1}$.  By Proposition \ref{HilbertCompleteIntersection}, $\Hilb^{m,ss}_{4,1} \subset U_m$, hence $\Hilb^m_{4,1} \gquot_{\Lambda_m} \SL(4) \cong U_m \gquot_{\Lambda_m \vert_{U_m}} \SL(4)$ by Lemma \ref{lemGIT}.  The rational map $\varphi_m : \mathbb PE \dashrightarrow \Hilb^m_{4,1}$ restricts to an isomorphism $\varphi_m \vert_U : U \to U_m$, and $\varphi_m \vert_U^* \Lambda_m = L(t) \vert_U$, where $t$ is given by the formula above.  It follows from Lem.~\ref{lemGIT} and Cor.~\ref{corWPE} that
$ \Hilb^m_{4,1} \gquot \SL(4) \cong U_m \gquot_{\Lambda_m \vert_{U_m}} \SL(4)  \cong U \gquot_{L(t) \vert_U} \SL(4) \cong \mathbb PE \gquot_t \SL(4)$.
\end{proof}

\section{Hassett--Keel Program}\label{sectm4}
So far, we have described the GIT quotients $\bP E\gquot_t\SL(4)$ parameterizing $(2,3)$-complete intersections in $\bP^3$, as well as the birational transformations among them as the linearization varies. To complete the proof of the Main Theorem stated in the introduction, we only need to relate these GIT quotients to the Hassett--Keel spaces $\overline{M}_4(\alpha)$. In fact, by \cite{g4ball} and \cite{maksymg4}, this is already known for the extremal values of the slope $t$ (see \eqref{end1} and \eqref{end2}). Now, using the GIT computation of the previous sections, we  will obtain in Theorem \ref{teopehk}  the   relationship for the intermediate cases.

 To prove the theorem, we will use some elementary properties of birational contractions (e.g.~\cite[\S1]{hukeel}).    Let $f:X \dashrightarrow Y$ be a birational map between normal projective varieties with $X$ $\mathbb{Q}$-factorial.  Let $(\pi_1,\pi_2):W\to X\times Y$  be a resolution of $f$, with $W$ projective (and $\pi_1$ birational).  We call $f$ a \emph{birational contraction} if every $\pi_1$-exceptional divisor is also a $\pi_2$-exceptional divisor.  In this case, for a $\mathbb Q$-Cartier divisor $D$ on $Y$, we define $f^*D$ to be $(\pi_1)_*(\pi_2^*D)$ and one can check that $H^0(Y,D)=H^0(X,f^*D)$.     These definitions are independent of the choice of resolution.

\begin{theorem}\label{teopehk}
Each of the log minimal models $\overline{M}_4 ( \alpha )$ for $\alpha \le\frac{5}{9}$ is isomorphic to one of the GIT quotients constructed above.  Namely, we have
$$ \overline{M}_4 ( \alpha ) \cong \bP E\gquot_t \SL(4) $$
where $t = \frac{34 \alpha - 16}{33 \alpha - 14}$ $\forall \alpha \in [\frac{8}{17} , \frac{5}{9} ]$.
\end{theorem}

\begin{proof}
We argue similarly to the case $\alpha = \frac{5}{9}$, which is Theorem 3.4 in \cite{g4ball}.  First, by the description of the GIT stability, we get that the natural map $$\varphi : \overline{M}_4 \dashrightarrow \bP E\gquot_t \SL(4)$$  is a birational contraction for all $t\in\left(0,\frac{2}{3}\right]$.  We then write
$$ \varphi^* (4s\eta + 4h) = a \lambda - b_0 \delta_0 - b_1 \delta_1 - b_2 \delta_2 .$$
(using $s=\frac{1}{t}$ and the scalar $4$ to make the formulas more attractive). The computations in \S \ref{ModuliSpace} tell us that $a = 34s-33$ and $b_0 = 4s-4$.  To compute the coefficients $b_1$ and $b_2$, we proceed exactly as in \cite{g4ball}.  In particular, let $Z \subset \overline{M}_4$ be the curve obtained by gluing a fixed non-hyperelliptic curve $C$ of genus 3 to a varying elliptic tail.  By the results of \S 1.3 of \cite{g4ball}, the map $\varphi$ is regular and constant along $Z$, so $b_1 = 14s-15$.  Specifically, we see that the image of $Z$ is the point corresponding to the cuspidal curve with normalization $C$, which is
$t$-stable for all $t\in(0,\frac{2}{3}]$.  Similarly, we see that if $j: \overline{M}_{2,1} \to \overline{M}_4$ is the standard gluing map, then $j^* \varphi^* (4s\eta + 4h)$ is supported along the union of $\delta_1$ and the Weierstrass divisor, and hence $b_2 = 18s-21$. In short, we obtain
$$
 \varphi^* (4s\eta + 4h) = (34s-33) \lambda - (4s-4) \delta_0 - (14s-15) \delta_1 - (18s-21) \delta_2.
$$

Now, since $\delta_1$ and $\delta_2$ are $\varphi$-exceptional and $t \le \frac{2}{3} < \frac{14}{17}$, we have
$$
H^0 ( \overline{M}_4 , n \varphi^* (4s\eta + 4h))
$$
\begin{eqnarray*}
&\cong& H^0 ( \overline{M}_4 , n \varphi^* (4s\eta + 4h) + (10s-11) \delta_1 + (14s-17) \delta_2 ))\\
 & =& H^0 ( \overline{M}_4 , n((34s-33) \lambda - (4s-4)( \delta_0 + \delta_1 + \delta_2 ))) .
 \end{eqnarray*}
Thus, for $s = \frac{33 \alpha - 14}{34 \alpha - 16}$,
\begin{eqnarray*}
 \bP E \gquot_t \SL(4) &=& \operatorname{Proj} \bigoplus_n H^0 (\mathbb PE\gquot_t \SL(4),n(4s\eta+4h))\\
 &=& \Proj \bigoplus_n H^0 ( \overline{M}_4 , n \varphi^* (4s\eta + 4h)) \\
& \cong& \operatorname{Proj} \bigoplus_n H^0 \left( \overline{M}_4 , n \left(K_{\overline{M}_4} + \alpha \delta \right)\right) \\
&=& \overline{M}_4 \left( \alpha \right).
\end{eqnarray*}
\end{proof}

\bibliography{g4bib}
\end{document}